\documentclass[12pt]{amsart}
\usepackage{pgfplots}
\usepackage{float}
\usepackage[utf8]{inputenc}
\usepackage{graphicx}
\usepackage{color}
\usepackage{mdwlist}
\usepackage{amssymb}
\usepackage{subfigure}
\usepackage{relsize}
\usepackage{graphics}
\allowdisplaybreaks

\usepackage{cleveref}

\def\g{\lambda}
\def\psiM{\psi_{\max}}
\def\psim{\psi_{\min}}
\def\rhoM{\rho_{\max}}
\def\rhom{\rho_{\min}}
\def\betam{\beta_{\min}}
\def\betaM{\beta_{\max}}
\def\SigmaL{\Sigma_L}

\newcommand{\beq}{\begin{eqnarray*}}
\newcommand{\feq}{\end{eqnarray*}}
\newcommand{\beqn}{\begin{eqnarray}}
\newcommand{\feqn}{\end{eqnarray}}

\newcommand{\RN}[1]{%
  \textup{\uppercase\expandafter{\romannumeral#1}}%
}

\newtheorem{theorem}{Theorem}[section]
\newtheorem*{theorem*}{Theorem}
\newtheorem{lemma}[theorem]{Lemma}
\newtheorem{corollary}[theorem]{Corollary}
\newtheorem{proposition}[theorem]{Proposition}
\theoremstyle{definition}

\theoremstyle{remark}
\newtheorem{remark}[theorem]{Remark}
\numberwithin{equation}{section}
\newcommand{\red}[1]{\textcolor{red}{#1}}
\newcommand{\blue}[1]{\textcolor{blue}{#1}}
\newcommand\numberthis{\addtocounter{equation}{1}\tag{\theequation}}

\pgfplotsset{compat=1.17}

\begin{document}
\title[The Euler-Poisson-alignment system]
{Critical thresholds in the Euler-Poisson-alignment system}% with a constant background}

\author{Manas Bhatnagar}
\author{Hailiang Liu}
\address{Department of Mathematics, Iowa State University, Ames, Iowa 50010}
\email{manasb@iastate.edu}
\email{hliu@iastate.edu}

\author{Changhui Tan}
\address{Department of Mathematics, University of South Carolina, Columbia, South Carolina 29208}
\email{tan@math.sc.edu}
% 35L65: Conservation laws,  35L67: Shocks and singularities
% 35B30: Dependence of solutions on initial and boundary data, parameters
% 35Q35: PDE in connection with fluid mechanics

\begin{abstract}  
This paper is concerned with the global wellposedness of the
Euler-Poisson-alignment (EPA) system. This system arises from
collective dynamics, and features two types of nonlocal interactions: the
repulsive electric force and the alignment force.
It is known that the repulsive electric force generates oscillatory
solutions, which is difficult to be controlled by the nonlocal
alignment force using conventional comparison principles.
We construct \emph{invariant regions} such that the solution trajectories cannot
exit, and therefore obtain global wellposedness for subcritical initial
data that lie in the invariant regions. Supercritical regions of
initial data are also derived which leads to finite-time singularity formations.
To handle the oscillation and the nonlocality,
we introduce a new way to construct invariant regions piece by piece in the phase
plane of a reformulation of the EPA system.
Our result is extended to the case when the alignment force is weakly
singular. The singularity leads to the loss of a priori bounds crucial in our analysis.
With the help of improved estimates on the nonlocal quantities, we
design non-trivial invariant regions that guarantee global
wellposedness of the EPA system with weakly singular alignment interactions.      
\end{abstract}
\maketitle

\section{Introduction}
\label{intro}
In this paper, the point of concern is the following one-dimensional
Euler-Poisson-alignment (EPA) system 
\begin{subequations}
\label{mainsys}
\begin{align}
& \rho_t + (\rho u)_x = 0, %\qquad t>0,\quad x\in\mathbb{T},
    \label{mainsysmass}\\
& u_t + u u_x = -k\phi_x +\int_{\mathbb{R}}\psi(x-y)(u(y)-u(x))\rho(y)dy, \label{mainsysmom}\\
& -\phi_{xx} = \rho - c, \label{mainsyspoisson}
\end{align}
\end{subequations}
%on a periodic spatial domain  $\mathbb{T} = [-\tfrac12, \tfrac12)$,
subject to smooth initial density and velocity 
$$
\big.(\rho(t,\cdot),u(t,\cdot))\big|_{t=0} = (\rho_0\geq 0,u_0).
$$

This system can be viewed as the pressureless Euler equations with two
types of nonlocal interacting forces on the right-hand side of the
momentum equation \eqref{mainsysmom}: the \emph{electric force} and
the \emph{alignment force}.

The electric force is modeled through an interacting protential
$\phi$, that is governed by the Poisson equation
\eqref{mainsyspoisson}, with a constant $c$ representing the
background charge that can be zero or a positive constant.
The parameter $k$ signifies the property of the underlying force:
repulsive $k>0$ or attractive $k<0$.
When only electric force is present, i.e. $\psi\equiv0$, \eqref{mainsys} reduces to the
classical \emph{Euler-Poisson system}.
It has been an area of intensive study due to their vast relevance in
modeling physical phenomena \cite{BRR94, HJL81, Ja75, Ma86, MaPe90,
  MRS90}, including semiconductor and plasma dynamics.

The alignment force describes the collective motion of an interacting
system, where the \emph{influence function} $\psi$ characterizes the
strength of the pairwise velocity alignment interaction.
Naturally, $\psi(x)=\psi(|x|)$ is assumed to be radial and decreasing
in $\mathbb{R}_+$.
When only alignment force is present, i.e. $k=0$,
the system reduces to the \emph{Euler-alignment system}, which serves
as a macroscopic realization of the celebrated agent-based
Cucker-Smale flocking model \cite{CS07a,CS07b}, c.f. \cite{HaTa08} for
a derivation.

% When no electric force is involved, i.e., $k=0$, the system reduces to the  Euler alignment equations. Euler alignment equations arise as macroscopic realization of agent-based dynamics \cite{CS07a,CS07b} which describes the collective motion of finite agents, each of which adjusts its velocity to a weighted average of velocities of its neighbors,
% \begin{align*}
% & \dot x_i =v_i, \\
% & \dot v_i= \frac{1}{N} \sum_{j=1}^N \psi(|x_i-x_j|)(v_j-v_i).
% \end{align*}
% Here $\psi$ is often called influence potential. % for model flocking of birds, fish and other organisms.
% See \cite{HaTa08} for realization of Euler alignment system as a mean field limit of the above type finite agent model as $N\to\infty$. 

The purpose of this work is to study the global regularity of the EPA system
\eqref{mainsys} for general initial data. It is well-known that the
finite-time breakdown of the pressureless Euler equations is generic,
see e.g. \cite{Lax64}. In particular, for all smooth initial data such
that $u_0$ is non-increasing, the solutions develop finite-time shock
formations. On the other hand, the interacting forces intend to help
avoiding the singularities.

For the 1D Euler-Poisson system with a repulsive force, a
\emph{critical threshold phenomenon}
is shown in \cite{ELT01}: there exists a large class of \emph{subcritical}
initial data that lead to global smooth solutions, while a class of
\emph{supercritical} initial data lead to finite-time shock
formations. See e.g. \cite{Lee2, LT02, LT03, TW08, Tan21} on
extensions to higher dimensions and with pressure.

For the Euler-alignment system, a similar critical threshold
phenomenon is observed in \cite{TT14} when the influence
function $\psi$ is bounded, c.f. also \cite{CCTT16, HT17}.
Recently, there is a growing interest on singular influence function
that are unbounded at the origin. When $\psi$ is \emph{strongly
  singular}, namely $\psi$ is non-integrable near the origin, it has
been shown in \cite{DKRT18} and \cite{ST17} independently that all
non-vacuous periodic initial data lead to global smooth solutions.
When $\psi$ is \emph{weakly singular}, namely unbounded but integrable
at the origin, critical thresholds are obtained in \cite{Tan20}, also see \cite{BLws} for improved bounds on density with any integrable $\psi$. 
For recent development on the Euler-alignment system, we refer readers
to the book \cite{Shv21} and the references therein.

For the EPA system \eqref{mainsys}, we expect the critical threshold
phenomenon when the influence function $\psi$ in the alignment force is
bounded. Such behavior has been first shown in \cite{CCTT16}, where
the Poisson equation \eqref{mainsyspoisson} is assumed to have a zero
background ($c=0$).
The authors in \cite{BL201} study the EPA system with attractive
electric forces ($k<0$) and nonzero, non-constant background ($c(x)>0$). The dynamics
are more subtle. They design highly non-trivial comparison principles
to take care of the nonlocality that arises from the alignment force, and
manage to obtain bounds on subcritical and supercritical regions of
initial data, thus describing the critical threshold phenomenon.

Our main focus of this paper is on the EPA system \eqref{mainsys}
where the electric forces is repulsive ($k>0$) and with non-zero
background ($c>0$).
This type of electric forces is physically relevant. The solution to
the corresponding Euler-Poisson system is known to generate
solutions that oscillate, e.g. \cite{ELT01}.
Such distinct feature makes it difficult to incorporate with the
nonlocal alignment forces. In particular, the comparison principles
used in \cite{BL201} are no longer valid. New analytical tools are
needed to capture the critical threshold phenomenon.

For convenience, we assume the spatial domain to be a torus
$\mathbb{T}=[-\tfrac12, \tfrac12)$, namely we consider $1$-periodic data.
We shall comment that many of our results can be extended to the
whole real line case with
$$
\int_{-\infty}^\infty (\rho_0(x)-c)dx = 0.
$$
We shall leave this case for future investigation.

Under the spatial domain $\mathbb{T}$, the Poisson equation
\eqref{mainsyspoisson} requires the background charge to be the
average density, that is conserved in time due to \eqref{mainsysmass}.
We have
\begin{equation}\label{eq:c}
  c=\int_{\mathbb{T}} \rho_0(x)dx.
\end{equation}
One useful parameter that plays an important role in quantifying the
strength of the electric force is
\begin{equation}\label{eq:gamma}
  \g=2\sqrt{\frac{k}{c}}.
\end{equation}
It is assumed to be a positive finite number throughout this paper.
The alignment force can be equivalently expressed as
\[\int_{\mathbb{T}}\psi_{\rm per}(y)(u(x+y)-u(x))\rho(x+y)\,dy,\]
with the periodic influence function
\[\psi_{\rm per}(x):=\sum_{m\in\mathbb{Z}}\psi(x+m),\quad\forall~x\in\mathbb{T},\]
which is symmetric with respect to zero.
We will continually use $\psi$ to represent the periodic influence
function for simplicity.

Our first main result is on the global wellposedness of the EPA system
\eqref{mainsys} with repulsive electric force $k>0$ and \emph{bounded} alignment influence:
\begin{equation}\label{eq:psibounded}
  0\le\psim\le\psi(x)\le\psiM,\quad\forall~x\in\mathbb{T}.
\end{equation}
We construct a class of subcritical initial data and show solutions
are globally regular; on the other hand, we also find a class of
supercritical initial data such that solutions experience finite-time
singularity formations.
The precise descriptions of such critical threshold phenomenon are
stated in Theorems \ref{gs1} and \ref{ftb1}.
Depending on the relative strength between the electric force and the
alignment force, there are three different scenarios:
(i). weak alignment ($\psiM<\g$), (ii). strong alignment
($\psim\ge\g$), and (iii). medium alignment ($\psim<\g\le\psiM$).
We construct subcritical regions $\Sigma_i$ and supercritical regions
$\Delta_i$ on initial data for each scenario, that leads to either
global wellposedness or finite-time blowup, respectively.

In particular, when the alignment force is weak or medium, the
solution is oscillatory. Instead of a direct comparison with an
auxiliary system, we construct an \emph{invariant region} in the phase
plane of the solutions along each characteristic path.
The novelty of our construction is that we use different auxiliary
systems to build segments of the boundary of the invariant regions,
and then glue them together. This allows us to handle the nonlocal
alignment force while the underlying Euler-Poisson system is highly
oscillatory.

We would like to point out a special case when $\psi$ is a constant,
known as \emph{all-to-all} alignment interactions. In this case, the
alignment force reduces to a local and linear damping, and
\eqref{mainsys} becomes the  \emph{damped
Euler-Poisson system}. 
The invariant regions that we constructed
are consistent with the sharp critical threshold conditions
obtained in \cite{BL19} on the damped Euler-Poisson system.

The next focus is on the singular alignment interactions.
When $\psi$ is strongly singular, the EPA system \eqref{mainsys} was
studied in \cite{KT18}. The surprising result indicates that the
alignment force dominates the electric force, regardless of whether
the electric force is attractive or repulsive. Any smooth non-vacuous
initial data lead to global smooth solutions.
The argument holds even if we drop the assumption $\psi\ge0$, namely
misalignment is allowed, as discussed in \cite{MTX21}.

Our second main result is on the EPA system \eqref{mainsys} with
repulsive electric force $k>0$ and \emph{weakly singular} alignment
influence:
\begin{equation}\label{eq:psiweaksing}
  \psi(x)\geq0,\quad\forall~x\in\mathbb{T},\quad\text{and}\quad
  \|\psi\|_{L^1(\mathbb{T})}<+\infty.
\end{equation}
In particular, $\psi$ can be unbounded at $x=0$.
Although the singularity is not strong enough to produce dominating
dissipation like the strongly singular case, the global behavior is
not expected to be the same as the case when $\psi$ is bounded.
Without the $L^\infty$ bound on $\psi$, we do not have the following a priori
bounds on the quantity $\psi\ast\rho$ (here $\ast$ denotes the
spatial convolution)
\begin{equation}\label{eq:psirho}
  \psi_m c\le \int_{\mathbb{T}}\psi(y)\rho(t,x-y)\,dy\le \psi_Mc,
  \quad\forall~t\geq0,
\end{equation}
which plays an essential role in the global
regularity of the Euler-alignment system (see \cite{Tan20}), as well
as our approach to the EPA system with bounded alignment interactions.

We construct a subcritical region on initial data such that the
solution is globally regular. The main idea is to replace \eqref{eq:psirho} by 
\begin{equation}\label{eq:psirhonew}
  \|\psi\|_{L^1}\rhom\le \int_{\mathbb{T}}\psi(y)\rho(t,x-y)\,dy\le \|\psi\|_{L^1}\rhoM,\quad\forall~t\geq0,
\end{equation}
where the bounds depend on the maximum and minimum of the solution $\rho$.
Then choose appropriate constants 
$\rhom$ and $\rhoM$, and build an invariant region that
is a subset of $\{\rho_0 : \rhom\le\rho_0(x)\le \rhoM\}$.
However, with the bound \eqref{eq:psirhonew}, we are not
able to obtain a non-trivial invariant region using our analytical
framework, with any choice of $\rhom$ and $\rhoM$.
Indeed, for the Euler-alignment system, it is observed in
\cite{Tan20} that, without the a priori bounds
like \eqref{eq:psirho}, additional treatments are required to control
$\rhoM$, and the critical threshold is different from the scenario
when $\psi$ is bounded. The presence of the electric force adds
another layer of complexity. 
To overcome such difficulty, we obtain refined bounds of
\eqref{eq:psirhonew}, stated in Lemma \ref{psistarnewboundslem},
making use of the equation \eqref{eq:c}.
With the refined bounds, we can obtain non-trivial invariant regions
by the right choices of $\rhom$ and $\rhoM$, and show global
regularity of the EPA system \eqref{mainsys} if initial data lie in
these subcritical regions. The precise statement is presented as Theorem
\ref{l1gs}. 

This paper is arranged as follows. Section \ref{mainresults} contains
the statements of the main results in this paper. Section
\ref{analysis1} entails the constructions of the
subcritical and supercritical regions for \eqref{mainsys} with bounded
alignment influence, proving Theorems \ref{gs1} and \ref{ftb1}. The
first three subsections focus on the subcritical regions to the three
different scenarios respectively. The fourth subsection is on the
construction of the supercritical regions. 
Section \ref{weaklysing} is devoted to the construction of the
invariant region for \eqref{mainsys} with weakly singular alignment
influence, proving Theorem \ref{l1gs}.

% \subsection*{Notations}
% Let $i = \sqrt{-1}$, $\betam: = M\min \psi$ and
% $\betaM: =  M\max\psi$. Note that $\betam$ could be
% zero. Let us also introduce the short-cut notations
% \[\tilde z := \sqrt{(4kc/\psi^{\ast 2}_{min}) - 1}\quad \text{and}
%   \quad \hat{z} := \sqrt{(4kc/\psi^{\ast 2}_{max}) - 1}.\]
% Note that $\tilde z, \hat z$ could be real, purely imaginary, as well as infinity.

\section{Main Results}\label{mainresults}
Let us start with a reformulation for the EPA system \eqref{mainsys} through an
auxiliary variable
\[G=u_x+\psi\ast\rho,\]
introduced in \cite{CCTT16}.
System \eqref{mainsys} can be expressed in the following equivalent form
\begin{subequations}
\label{localexeqsys}
\begin{align}
& G_t + (Gu)_x = k(\rho -c), \label{localexeqsys2}\\
& \rho_t + (\rho u)_x = 0, \label{localexeqsys1}\\
& u_x = G - \psi\ast\rho. \label{localexeqsys3}
\end{align}
\end{subequations}
The velocity $u$ can be recovered from \eqref{localexeqsys3}. It is
uniquely defined up to a constant shift. The constant can then be
uniquely determined by the total momentum $\int_{\mathbb{T}}\rho u
dx$, which is conserved in time.

We state a local wellposedness result for smooth solutions to
\eqref{localexeqsys}. The proof can be done using energy estimates on
the derivatives of $(G,\rho)$. See \cite[Theorem 2.1]{Tan20} for a
complete proof when $k=0$.
The result can be easily extended to the case when $k\neq0$, c.f. also
\cite{BL202,KT18}.

% We can obtain energy estimates on derivatives of $\rho,G$. $u$ can be recovered in terms of $G,\psi\ast\rho$. This has been done recently in several works, see for example \cite{BL202,Tan20}. In particular, we can show the following,
% $$
% Z(t)\leq C(1+ ||G(t,\cdot)||_\infty + ||\rho(t,\cdot)||_\infty) Z(0), \quad t\in [0,T],
% $$
% where $Z(t) = ||G(t,\cdot)||_{H^s}^2 + ||\rho(t,\cdot)||_{H^s}^2$,
% $s\geq 1$. Such an energy inequality allows us to extend solutions
% to arbitrary time as long as the infinity norms of $G,\rho$ are
% bounded. We only state the result here.

\begin{theorem}[Local wellposedness]
\label{local}
Consider the system \eqref{localexeqsys} with initial data
\begin{equation}\label{eq:initHs}
  G_0 \in H^s( \mathbb{T}),\quad s>\tfrac12,\quad
  \rho_0\in (L^1_+\cap H^s)(\mathbb{T}),
\end{equation}
and interactions with $k\in\mathbb{R}$, $\psi\in L^1(\mathbb{T})$.
Then, there exists a time $T>0$ such that the solution
\[ G \in C\big([0,T]; H^s(\mathbb{T})\big),\quad
\rho \in C\big([0,T]; (L^1_+\cap H^s)(\mathbb{T})\big).\]
Consequently, the EPA system \eqref{mainsys} has a smooth solution
\[\rho \in C\big([0,T]; (L^1_+\cap H^s)(\mathbb{T})\big),\quad
  u \in C\big([0,T]; H^{s+1}(\mathbb{T})\big).\]
Moreover, $T$ can be extended as long as 
\begin{equation}\label{BKM}
  \int_0^T \Big(\|G(t,\cdot)\|_{L^\infty} + \|\rho(t,\cdot)\|_{L^\infty}
  \Big) dt < \infty.
\end{equation}
\end{theorem}
% \begin{remark}
% For systems such as the Euler alignment system with bounded kernel, it is known that shock formation ($|u_x|\to\infty$) and density concentration ($\rho\to\infty$) happens at the same time and at the same point in space. Therefore, one essentially only needs to bound the derivative of $u$ to extend local solutions to further times. However, for weakly singular alignment kernels, it can be shown that density concentrates before shock forms, see \cite{Tan20}. Indeed the inequality for the quantity $Z$ above confirms that in general, we need bounds on $||u_x||_\infty$ as well as $||\rho||_\infty$. 
% \end{remark}

The regularity criterion \eqref{BKM} indicates: the
global-in-time bounds on $G$ and $\rho$ are sufficient to obtain global
regularity.

Our first main result focuses on repulsive electric force $\g>0$ and
bounded influence functions $\psi$ in the alignment force \eqref{eq:psibounded}.

\begin{theorem}[Global solutions]
\label{gs1}
Consider \eqref{localexeqsys} with repulsive electric force $k>0$ and
bounded alignment influence $\psi$ satisfying \eqref{eq:psibounded}.
Suppose the initial data $(G_0, \rho_0)$ satisfies \eqref{eq:initHs}.
Then
\begin{enumerate}
    \item 
    \emph{Weak alignment} ($\psiM<\g$):
    under the admissible condition 
    \begin{equation}\label{eq:weakacc}
    \psiM - \psim < \frac{e^{\frac{\tan^{-1}\hat z}{\hat z}}\left(1 - e^{ -\frac{\pi}{\tilde z} - \frac{\pi}{\hat z} }\right) }{2\left(1 + e^{-\frac{\pi}{\tilde z}}\right)}\g,
    \end{equation}
    if the initial data lie in the subcritical region $\Sigma_1$, namely
    $$
    \big(G_0(x),\rho_0(x)\big)\in\Sigma_1, \quad \forall\,x\in\mathbb{T},
    $$
    then $(G, \rho)$ remain bounded in all time.
    \item
    \emph{Strong alignment} ($\psim\ge\g$):
    if the initial data lie in the subcritical region $\Sigma_2$, namely
    $$
   \big(G_0(x),\rho_0(x)\big)\in\Sigma_2, \quad \forall\,x\in\mathbb{T},
    $$
    then $(G, \rho)$ remain bounded in all time.
    \item
    \emph{Medium alignment} ($\psim<\g\le \psiM$):
    under the admissible condition 
    \begin{equation}\label{eq:medacc}
   \psiM - \psim < \frac{e^{\frac{\tan^{-1}\hat z}{\hat z}}}{2\left(1 + e^{-\frac{\pi}{\tilde z}} \right)}\g,
    \end{equation}
    if the initial data lie in the subcritical region $\Sigma_3$, namely
    $$
   \big(G_0(x),\rho_0(x)\big)\in\Sigma_3, \quad \forall\, x\in\mathbb{T},
    $$
    then $(G, \rho)$ remain bounded in all time.
\end{enumerate}
Consequently, \eqref{localexeqsys} has a global smooth solution.
Here, the parameters $\hat z$ and $\tilde z$ are defined as
\begin{equation}\label{eq:z}
  \hat z := \sqrt{\left(\frac{\g}{\psiM}\right)^2-1}
  \quad\text{and}\quad
  \tilde z := \sqrt{\left(\frac{\g}{\psim}\right)^2-1}.
\end{equation}
Note that $\hat z$, $\tilde z$ could be real, purely imaginary, as
well as infinity.
The regions $\Sigma_1, \Sigma_2$ and $\Sigma_3$ are subsets of
$\mathbb{R}\times\mathbb{R}_+$, defined in
\eqref{invregtransf}, \eqref{invregtransf2} and \eqref{invregtransf3}
respectively.
\end{theorem}

\begin{remark}
  The subcritical regions $\Sigma_1, \Sigma_2$ are
  illustrated in Figure \ref{mainfig}. The shape of $\Sigma_3$ is
  similar to $\Sigma_1$.
  We would like to point out that the steady-state solution
  $(G,\rho) = ( c\|\psi\|_{L^1} ,c)$ to \eqref{localexeqsys} is
  included in the subcritical regions $\Sigma_1$, $\Sigma_2$ and
  $\Sigma_3$. This corresponds to the steady-state solution
  $\rho(x)\equiv c$ and $u(x)\equiv\bar{u}$ to \eqref{mainsys}.
  Therefore, our subcritical regions are non-empty, and contains a large class of
  physically meaningful initial data, including the states around a
  steady state.
\end{remark}

\begin{remark}\label{rem:acc}
  When $\psi(x)\equiv\psi$ is a constant, the alignment force becomes
  a local and linear damping. Our constructed invariant regions agree
  with the sharp subcritical threshold obtained in \cite{BL19}.
  The admissible conditions \eqref{eq:weakacc} and \eqref{eq:medacc}
  automatically hold. For general $\psi$, the admissible conditions
  ensures the nonlocality is not too strong, and the invariant regions
  are non-trivial.
\end{remark}

\begin{theorem}[Finite time breakdown]
\label{ftb1}
Under the same assumptions as Theorem \ref{gs1}, we have
\begin{enumerate}
    \item
      \emph{Weak alignment} ($\psiM<\g$):
      If there exists $x_0\in\mathbb{T}$ that lie in the supercritical
      region $\Delta_1$, namely 
      $$
      \big(G_0(x_0),\rho_0(x_0)\big)\in\Delta_1,
      $$
      then $(G,\rho)$ becomes unbounded at a finite time.
    \item
      \emph{Strong and medium alignment} ($\psiM\geq\g$):
      If there exists $x_0\in\mathbb{T}$ that lie in the supercritical
      region $\Delta_2$, namely 
    $$
        \big(G_0(x_0),\rho_0(x_0)\big)\in\Delta_2,
    $$
    then $(G,\rho)$ becomes unbounded at a finite time.
\end{enumerate}
Moreover, at the blowup time $t_c$ and location $x_c$, the solution
generate a singular shock, with
\[\lim_{t\to t_c^-}  \rho(t,x_c) = \infty \,\,\text{or}\,\,0,\quad
  \lim_{t\to t_c^-} G(t,x_c) = -\infty,\quad
  \lim_{t\to t_c^-}  u_x(t,x_c ) = -\infty.\]
The regions $\Delta_1, \Delta_2$ are defined in \eqref{ftbinvregtransf}, \eqref{ftbinvregtransf2} respectively.
\end{theorem}

\begin{figure}[ht!] 
\centering
\subfigure[Weak alignment ($\lambda = \sqrt{2}, \psiM = 0.75,\psim = 0.25 $)]{\label{subcrGrhoweakinitial} \includegraphics[width=0.4\linewidth]{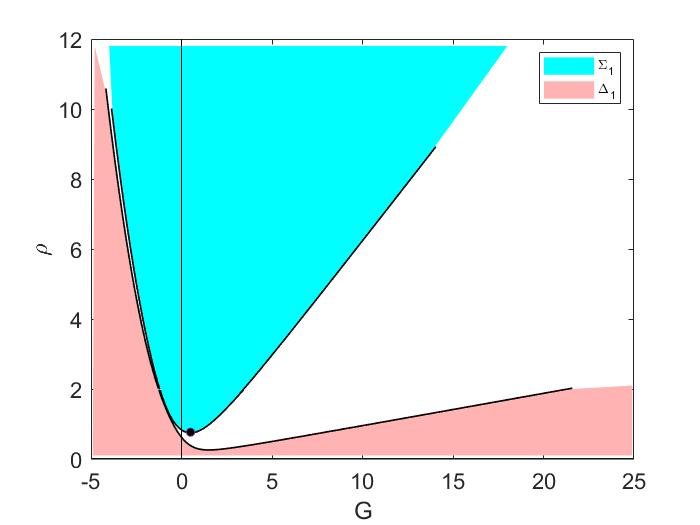}}%
\qquad
\subfigure[Strong alignment ($\lambda = \sqrt{2}, \psiM = 2,\psim = 1.5 $)]{\label{subcrGrhostronginitial}\includegraphics[width=0.4\linewidth]{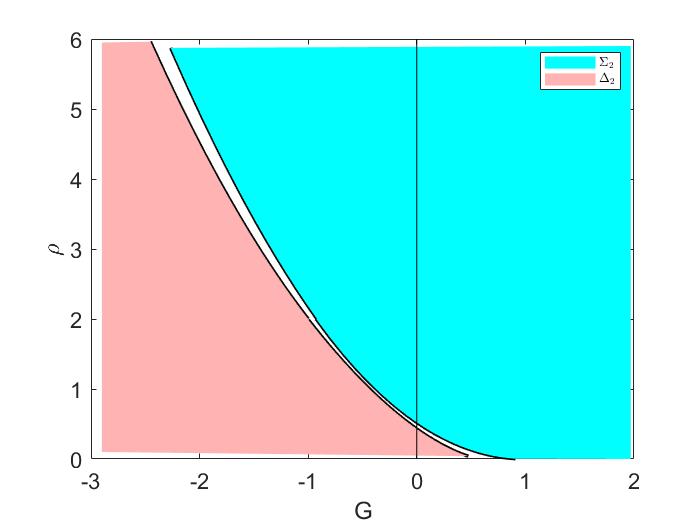}}
\caption{Shapes of $\Sigma_1,\Sigma_2,\Delta_1,\Delta_2$.}
\label{mainfig}
\end{figure}

\if 0
\begin{remark} When the momentum equation is augmented with an additional damping term, i.e., 
$$
u_t + u u_x +\nu u  = -k\phi_x + \psi\ast(\rho u) -  u\psi \ast \rho, 
$$
a similar result can still be obtained. 
%\begin{subequations}
%\begin{align}
%& G' = -G(G- \nu -\psi\ast\rho) + k(\rho -c) \label{Geq}\\
%& \rho' = \rho (G - \nu -\psi\ast\rho). \label{rhoeq}
%\end{align}
%\end{subequations}
Indeed, in our analysis we only need to set $w:= \frac{ u_x  +\nu+  \psi\ast\rho}{\rho}$ and $s:=\frac{1}{\rho}$ to get,
\begin{align*}
& w' = k - kcs,\\
& s' = w - s (\nu + \psi\ast\rho) . 
\end{align*}
In the result (Theorems \ref{gs1} and \ref{ftb1}), the bounds for $\psi$ need to be replaced by the same bounds added to $\frac{\nu}{M}$, %i.e., $\betam:= M\min\psi + \nu$ and $\betaM:= M\max\psi + \nu$.
i.e., $\min\psi, \max\psi$ should be replaced by $\min\psi +\nu/M, \max\psi + \nu/M$ respectively.
\end{remark}
\fi

Our second main result concerns the EPA system with weakly singular
alignment influence \eqref{eq:psiweaksing}. Although one would expect
a similar critical threshold phenomenon for the global behaviors of
the solutions, the lack of boundedness on $\psi$ would yield a lack of
apriori control on $\psi\ast\rho$, resulting a different subcritical
region for global smooth solutions.

\begin{theorem}[On weakly singular alignment force]  \label{l1gs}
Consider \eqref{localexeqsys} with repulsive electric force $k>0$ and
weakly singular alignment influence $\psi$ satisfying \eqref{eq:psiweaksing}.
Suppose the initial data $(G_0,\rho_0)$ satisfies \eqref{eq:initHs}. Then
\begin{enumerate}
    \item \emph{Weak alignment} ($\|\psi\|_{L^1}-\gamma  < \tfrac{\g}{2}$): under the admissible condition
    \begin{equation}\label{eq:wscond}
  4(\|\psi\|_{L^1} - 2\gamma) < \frac{e^{\frac{\tan^{-1}\hat z}{\hat z}}\left(1 - e^{ -\frac{\pi}{\tilde z} - \frac{\pi}{\hat z} }\right) }{2\left(1 + e^{-\frac{\pi}{\tilde z}}\right)}\g,
  \end{equation}
  if the initial data lie in the subcritical region $\SigmaL^1$,
namely
\[
 \big(G_0(x),\rho_0(x)\big)\in\SigmaL^1, \quad \forall\,x\in\mathbb{T},
\]
then $(G,\rho)$ remain bounded in all time.
      \item \emph{Strong alignment} ($\gamma  \geq \tfrac{\g}{2}$): if the initial data lie in the subcritical region $\SigmaL^2$, namely
    $$
   \big(G_0(x),\rho_0(x)\big)\in\SigmaL^2, \quad \forall\,x\in\mathbb{T},
    $$
    then $(G, \rho)$ remain bounded in all time.
    \item
    \emph{Medium alignment} ($\gamma<\frac{\g}{2}\le ||\psi||_1-\gamma$):
    under the admissible condition 
    \begin{equation}\label{eq:wsmedacc}
   4(||\psi||_{L^1}-2\gamma) < \frac{e^{\frac{\tan^{-1}\hat z}{\hat z}}}{2\left(1 + e^{-\frac{\pi}{\tilde z}} \right)}\g,
    \end{equation}
    if the initial data lie in the subcritical region $\SigmaL^3$, namely
    $$
   \big(G_0(x),\rho_0(x)\big)\in\SigmaL^3, \quad \forall\, x\in\mathbb{T},
    $$
    then $(G, \rho)$ remain bounded in all time.
\end{enumerate}
Consequently, \eqref{localexeqsys} has a global smooth solution.
Here, $\gamma = \int_{1/2}^1\psi^*(x)\, dx$, where $\psi^* :
(0,1]\to\mathbb{R}$ is the decreasing rearrangement of $\psi$ on
$\mathbb{T}$.
The parameters $\hat z$ and $\tilde z$ are defined as
\begin{equation}\label{eq:zl1}
  \hat z := \sqrt{\left(\frac{\g}{2(\|\psi\|_{L^1}-\gamma)}\right)^2-1}
  \quad\text{and}\quad
  \tilde z := \sqrt{\left(\frac{\g}{2\gamma}\right)^2-1}.
\end{equation}
The regions $\SigmaL^1,\SigmaL^2$ and $\SigmaL^3$ are subsets of $\mathbb{R}\times\mathbb{R}_+$ defined in \eqref{l1subcr}, \eqref{l1subcr2} and \eqref{l1subcr3} respectively.
\end{theorem}

\begin{remark}
  Unlike the case when $\psi$ is bounded, the subcritical regions
  $\SigmaL^i$'s are subsets of $\{(G_0,\rho_0) :
  \rhom\le\rho_0\le\rhoM\}$ for appropriate choices of
  $0\le\rhom<c<\rhoM<\infty$. Figure \ref{l1subcrfig} illustrates the
  shape of $\SigmaL^1$ and $\SigmaL^2$.
  The steady-state solution $(G,\rho)=(c\|\psi\|_{L^1},c)\in\SigmaL^i$.
  Hence, the region $\SigmaL^i$ contain initial data around the steady state. 
\end{remark}

\begin{figure}[ht!] 
\centering
\subfigure[Weak alignment ($\lambda = 4, \|\psi\|_{L^1} = 2, \gamma=0.95$)]{\label{l1subcrfig} \includegraphics[width=0.4\linewidth]{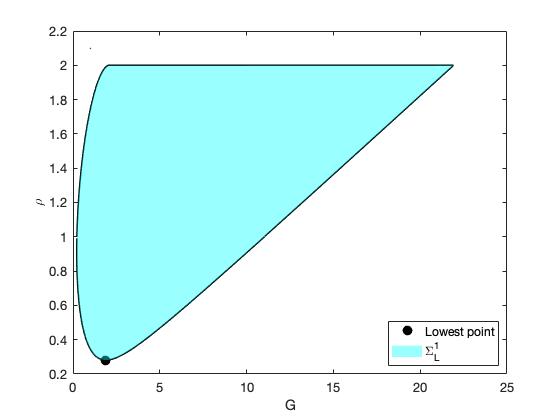}}%
\qquad
\subfigure[Strong alignment ($\lambda = \sqrt{2}, \|\psi\|_{L^1} = 2, \gamma=0.95 $)] {\label{wssalsubcr}\includegraphics[width=0.4\linewidth]{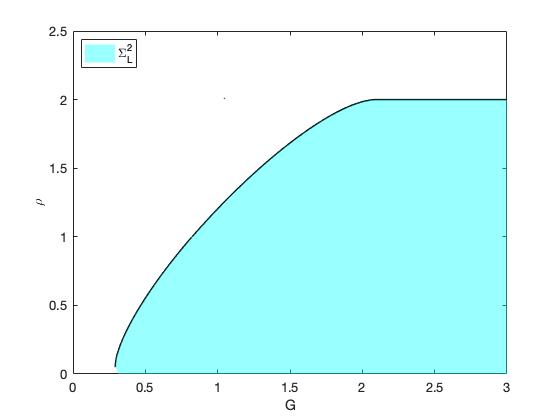}}
\caption{Shapes of $\SigmaL^1,\SigmaL^2$.}
\end{figure}

\begin{remark}
  The admissible conditions \eqref{eq:wscond} and \eqref{eq:wsmedacc} are similar to
  \eqref{eq:weakacc} and \eqref{eq:medacc} respectively. Since $\psi$ is unbounded, $\psiM-\psim$ is
  replaced by $4(\|\psi\|_{L^1}-2\gamma)$. Note that
  $\|\psi\|_{L^1}-2\gamma\ge0$, and the equality holds if and only if
  $\psi(x)\equiv\psi$ is a constant. Hence, just like the comment in
  Remark \ref{rem:acc}, the admissible condition says that the
  nonlocality is not too strong. The parameters $\hat z$ and $\tilde
  z$ are also revised to adapt the unboundedness of $\psi$.
\end{remark}

\section{The EPA system with bounded alignment influence}
\label{analysis1}
Consider the characteristic path $x(t)$ originated at $\alpha\in\mathbb{T}$
\begin{equation}
\label{chpath}
\frac{dx}{dt} = u(t,x(t)),\quad x(0) = \alpha.
\end{equation}
From \eqref{localexeqsys1} and \eqref{localexeqsys2}, we obtain the system
\begin{subequations}
\label{Grhosys}
\begin{align}
& G' = -G(G -\psi\ast\rho) + k(\rho -c),  \label{Geq}\\
& \rho' = -\rho (G -\psi\ast\rho), \label{rhoeq}
\end{align}
\end{subequations}
with initial data $G(0) =G_0(\alpha)$ and $\rho(0)= \rho_0(\alpha)$.
Here $'$ denotes the derivative along the characteristic path
\[f'(t)=\frac{d}{dt}f(t,x(t))=f_t(t,x(t))+u(t,x(t)) f_x(t,x(t)).\]

In the proofs of Theorems \ref{gs1} and \ref{ftb1}, we will justify
that the initial data when $\rho(0)=0$ can be handled separately. For
now, we assume that $\rho(0)>0$.  We can further apply the transformation
\begin{equation}
\label{vartransf}
w:= \frac{G}{\rho},\qquad s:=\frac{1}{\rho}
\end{equation}
to \eqref{Grhosys} and obtain the dynamics
\begin{subequations}
\label{mainodesys}
\begin{align}
& w' = k - kcs, \label{mainodesysa}\\
& s' = w - s (\psi\ast\rho) . \label{mainodesysb}
\end{align}
\end{subequations}
%Let $M\max \psi  $ and $M\min \psi = \betam$.
This ODE system is not closed along each characteristic path due to
the nonlocal nature of the term $\psi*\rho$. We shall analyze this nonlocal system by establishing a type of comparison argument.  To this end,  
we introduce a family of auxiliary systems
\begin{subequations}
\label{auxsys}
\begin{align}
& p' = k - kcq, \label{auxsysa}\\
& q' = p - \beta q, \label{auxsysb}
\end{align}
\end{subequations}
with $p=p(t;\beta), q=q(t;\beta)$, where $\beta$ is a  parameter.
For each given $\beta$, \eqref{auxsys} is a linear system that can be
solved explicitly. We can rewrite \eqref{auxsys} as
\[
\begin{bmatrix}
p - \frac{\beta}{c} \\
q - \frac{1}{c}
\end{bmatrix}' = 
\begin{bmatrix} 0 & -kc\\
1 & -\beta
\end{bmatrix}
\begin{bmatrix}
p - \frac{\beta}{c} \\
q - \frac{1}{c}
\end{bmatrix},
\]
where the coefficient matrix has two eigenvalues 
\[\frac{-\beta \pm \sqrt{\beta^2-4kc}}{2}.\]
Note that $\psi\ast\rho$ has apriori bounds \eqref{eq:psibounded}, which
we recall here: 
$\betam\leq \psi\ast\rho\leq \betaM$, where we denote
\begin{equation}\label{eq:betaMm}
  \betaM=c\psiM,\quad\betam=c\psim.
\end{equation}
It is natural to consider the following two particular auxiliary systems with $\beta =
\betaM$ and $\betam$:
\begin{subequations}
\label{mumMaux}
\begin{align}
& \begin{bmatrix}
\hat p - \frac{\betaM}{c}\\
\hat q - \frac{1}{c}
\end{bmatrix}' = 
\begin{bmatrix}
0 & -kc\\
1 & -\betaM
\end{bmatrix}
\begin{bmatrix}
\hat p - \frac{\betaM}{c}\\
\hat q - \frac{1}{c}
\end{bmatrix}, \label{muMaux}\\
& \begin{bmatrix}
\tilde p - \frac{\betam}{c} \\
\tilde q - \frac{1}{c}
\end{bmatrix}' = 
\begin{bmatrix}
0 & -kc\\
1 & -\betam
\end{bmatrix}
\begin{bmatrix}
\tilde p - \frac{\betam}{c}\\
\tilde q - \frac{1}{c}
\end{bmatrix}. \label{mumaux}
\end{align}
\end{subequations}

We would like to remark that there is no direct comparison principle between the
solutions to the nonlocal system $(w(t), s(t))$ and the local auxiliary
system $(\hat p(t), \hat q(t))$ or $(\tilde p(t), \tilde q(t))$,
particularly when $\beta$ is small, in which case the eigenvalues are
not real, and the solutions are oscillatory.
Instead, we shall obtain a comparison in the phase plane, and obtain
an \emph{invariant region} that the trajectory $(w,s)$ cannot exit.

\subsection{Weak alignment}
\label{weakal}
We begin with the case where all admissible values of $\beta \in [\betam,\betaM]$ are such that 
$$
\beta^2 <4kc,
$$
%strictly less than $4kc$. 
and in such case $(\beta/c,1/c)$ is an asymptotically stable spiral point.
Physically, this places a restriction on the upper bound of $\psi\ast\rho$. 
Hence, we call this scenario the weak alignment case. 
%To this end, assume $4kc >  \psi^{\ast 2}_{max}$.
We will construct an invariant region using specific trajectories of the above auxiliary systems, see Figure \ref{fig1}. At this point, we establish some notation to be used in this section,
$$
\tilde \theta: = \frac12\sqrt{4kc-\betam^2}, \quad
\hat \theta:= \frac12\sqrt{4kc-\betaM^2}.
$$
%Following three IVPs are for three continuous curve which are segments of the trajectories of \eqref{muMaux} and \eqref{mumaux} on the phase plane $p-q$.
\iffalse
Here,
\begin{subequations}
\begin{align}
& p_1^\ast = \sqrt{\frac{k}{c}}\left( \frac{\mu_M}{\sqrt{kc}} - e^{\frac{\mu_M}{2\theta_M}\tan^{-1}\frac{2\theta_M}{\mu_M}} \right),\quad \theta_M = \frac{\sqrt{4kc - \mu_M^2}}{2},\\
& p_2^\ast = \frac{\mu_m}{c}\left( 1 + e^{\frac{\mu_m \pi}{2\theta_m}}\right) - p_1^\ast e^{\frac{\mu_m \pi}{2\theta_m}},\quad \theta_m = \frac{\sqrt{4kc - \mu_m^2}}{2},\\
& p_3^\ast = \frac{\mu_M}{2c} + \left( p_2^\ast - \frac{\mu_M}{c} \right) e^{-\frac{\mu_M t^\ast}{2}}\cos (\theta_M t^\ast),\\
& \theta_M e^{\frac{\mu_M t^\ast}{2}} + (cp_2^\ast - \mu_M)\sin (\theta_M t^\ast) = 0, \quad \theta_M t^\ast\in \left(-\pi + \tan^{-1}\left( \frac{2\theta_M}{\mu_M} \right) , 0\right).\nonumber
\end{align}
\end{subequations}
The justification to above definitions of $p_i^\ast$ will follow soon. Define the region,
\begin{align}
\label{subcr}
& \tilde\Sigma = \{ (p,q): \eta_1(p) < q < \eta_2(p),\ p\in (p_1^\ast , 0 ) \}\\
& \qquad \cup \{ (p,q): 0 < q < \eta_2(p),\ p\in (0 , p_3^\ast ) \} \nonumber\\
& \qquad \cup \{ (p,q): \eta_3(p) < q < \eta_2(p),\ p\in (p_3^\ast , p_2^\ast ) \}. \nonumber
\end{align}
\fi

\begin{figure}[ht!] 
\centering
\subfigure{\includegraphics[width=0.5\linewidth]{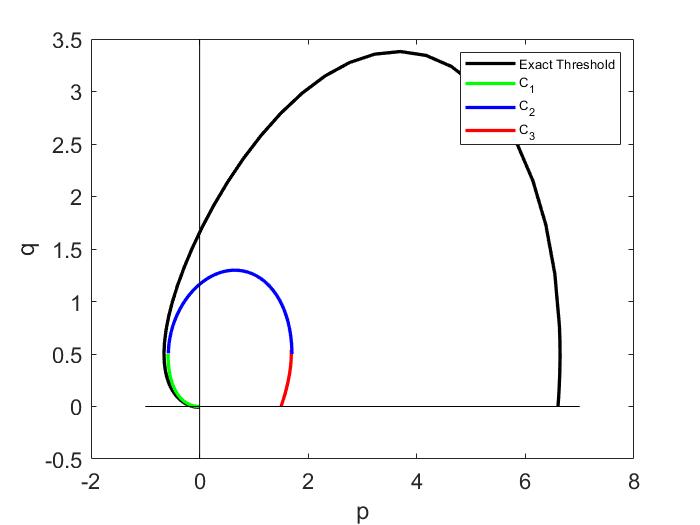}}
\caption{Invariant region.}
\label{fig1}
\end{figure}
We will now construct the invariant region ($\Sigma_1^\ast$) as in Figure \ref{fig1}. We divide this construction into three steps, each pertaining to one of the segments of the boundary of $\Sigma_1^\ast$. We will start from the origin and move backwards in time. 

{\bf Step 1:} The first segment of the curve is the trajectory to \eqref{muMaux} starting at the origin, going into the second quadrant, and ending when it hits the line $q=1/c$ while going backwards in time. Hence, if we solve for $\hat p,\hat q$ with $\hat p(0)=0, \hat q(0)=0$, then the other end point of the curve is $(\hat p(t_1^{we}),\hat q (t_1^{we}))$, where $t_1^{we}$ is the first negative time for which $\hat q(t_1^{we} ) = 1/c$. Let $p_1^{we} := \hat p(t_1^{we})$. Since \eqref{muMaux} is a simple linear system, we can explicitly solve for its solution with initial data $(\hat p(0),\hat q(0))=(0,0)$,
\begin{align*}
& \hat p(t) = \frac{\betaM}{c} + e^{-\frac{\betaM t}{2}}\left( -\frac{\betaM}{c} \cos\hat\theta t + \frac{k-\frac{\betaM^2}{2c}}{\hat \theta}\sin\hat\theta t \right),\\
& \hat q(t) = \frac{1}{c} - \frac{e^{-\frac{\betaM t}{2}}}{c}\left( \cos\hat\theta t + \frac{\betaM}{2\hat\theta}\sin\hat\theta t \right).
\end{align*}
$$
\hat q(t_1^{we})=\frac1c\implies t_1^{we} = -\frac{1}{\hat\theta} \tan^{-1}\left(\frac{2\hat\theta}{\betaM}\right).
$$
Hence,
\begin{align}
p_1^{we} & = \frac{\betaM}{c} - \sqrt{\frac{k}{c}} e^{\frac{\betaM}{2\hat\theta}\tan^{-1}\left( \frac{2\hat\theta}{\betaM} \right)} \label{p1ast}\\
& = \frac{\betaM}{c} - \sqrt{\frac{k}{c}} e^{\frac{\tan^{-1}( \hat z )}{\hat z}}, \nonumber
\end{align}
with $\hat z$ as defined at the end of Section \ref{intro}.
\begin{lemma}
\label{p1astbound}
$p_1^{we}\in \sqrt{\frac{k}{c}}\,( -1, 2-e )$.
\end{lemma}
\begin{proof}
We write the expression for $p_1^{we}$ as,
$$
f(\tau) = \sqrt{\frac{k}{c}}\left( \frac{1}{\tau} - e^{\frac{\tan^{-1}\sqrt{4\tau^2-1}}{\sqrt{4\tau^2 - 1}}} \right),\quad \tau=\frac{\sqrt{kc}}{\betaM}>\frac{1}{2}.
$$
One can evaluate that $f$ is a monotonically decreasing function with
\[\lim_{\tau\to \frac12^+}f(\tau) = (2-e)\sqrt{\frac kc}\quad\text{and}\quad
  \lim_{\tau\to\infty}f(\tau) = -\sqrt{\frac kc}.\]
Hence, the result holds.
\end{proof}

Also note that $\hat p'>0$ and $\hat q'<0$ in this region. Hence, the first segment is given by,
\begin{align}
\label{firstseg}
& C_1 = \{ (p,q): (\hat p(t),\hat q(t)), t\in [t_1^{we} , 0]\} .
\end{align}
%Also, the first segment $(q=\eta_1(p))$ of the boundary of invariant region satisfies the IVP \eqref{eta1} since it is the trajectory equation.

{\bf Step 2:} The second segment is constructed using the trajectory of the system \eqref{mumaux}. To have a closed region, the starting point of this segment should be the endpoint of the first segment. To this end, let $\tilde p,\tilde q$ be solutions to IVP \eqref{mumaux} with $\tilde p(0) = p_1^{we}, \tilde q(0) = 1/c$. This segment starts at $(p_1^{we},1/c)$, traces the trajectory of $(\tilde p,\tilde q)$ upwards and ends when it hits the $q=1/c$ line again in the first quadrant. We denote the end point as $(p_2^{we} ,1/c)$. In particular, $p_2^{we} = \tilde p(t_2^{we} )$ where $t^{we}_2$ is the first negative time where $\tilde q(t_2^{we}) = 1/c$. We have
\begin{align*}
& \tilde p(t) = \frac{\betam}{c}+ e^{-\frac{\betam t}{2}}\left[\left(p_1^{we} - \frac{\betam}{c}\right)\cos\tilde\theta t + \left( \frac{p_1^{we}\betam}{2\tilde\theta} - \frac{\betam^2}{2 c\tilde\theta} \right)\sin\tilde\theta t \right],\\
& \tilde q(t) = \frac{1}{c} + \frac{e^{-\frac{\betam t}{2}}}{\tilde\theta}\left(p_1^{we} - \frac{\betam}{c}\right)\sin\tilde\theta t .
\end{align*}
$$
\tilde q(t_2^{we}) = \frac1c\implies t_2^{we} = -\frac{\pi}{\tilde\theta}.
$$
Consequently,
\begin{align*}
p_2^{we} = \tilde p(t_2^{we}) & = \frac{\betam}{c} - e^{\frac{\betam\pi}{2\tilde\theta}}\left(p_1^{we} - \frac{\betam }{c}\right)\\
& = \frac{\betam}{c}\left( 1 + e^{\frac{\betam\pi}{2\tilde\theta}} \right) - p_1^{we} e^{\frac{\betam\pi}{2\tilde\theta}} \numberthis  \label{p2ast} \\
& = \frac{\betam}{c}\left( 1 + e^{\frac{\pi}{\tilde z}} \right) - p_1^{we} e^{\frac{\pi}{\tilde z}}.
\end{align*}
%From \eqref{p1astbound},
%\begin{align}
%\label{p2astbound}
%& \frac{\mu_m}{c}\left( 1 + e^{\frac{\mu_m\pi}{2\theta_m}} \right) + (e-2)\sqrt{\frac{k}{c}} e^{\frac{\mu_m\pi}{2\theta_m}} < p_2^\ast < \frac{\mu_m}{c}\left( 1 + e^{\frac{\mu_m\pi}{2\theta_m}} \right) + \sqrt{\frac{k}{c}}e^{\frac{\mu_m\pi}{2\theta_m}}.
%\end{align}
Here, we emphasize an important issue. We must have that $p_2^{we} > \betaM/c$, for otherwise we would not be able to obtain a closed invariant region. The following Lemma states a condition to ensure this. 
\begin{lemma}
\label{clocon1lem}
$p_2^{we}>\betaM/c$ if and only if 
\begin{align*}
& \sqrt{kc} e^{\frac{\tan^{-1}\hat z}{\hat z}} > (\betaM-\betam)\left( 1 + e^{-\frac{\pi}{\tilde z}} \right). 
\end{align*}
\end{lemma}
\begin{proof}
From \eqref{p1ast} and \eqref{p2ast}, we have
\begin{align*}
 p_2^{we} - \frac{\betaM}{c} & = \frac{\betam}{c}\left( 1 + e^{\frac{\pi}{\tilde z}} \right) - p_1^{we} e^{\frac{\pi}{\tilde z}} - \frac{\betaM}{c} \\
 & =   \frac{\betam}{c} e^{\frac{\pi}{\tilde z}} - \left( \frac{\betaM}{c} - \sqrt{\frac{k}{c}} e^{\frac{\tan^{-1}(\hat z)}{\hat z}} \right) e^{\frac{\pi}{\tilde z}} - \frac{\psi^\ast_{ max} - \betam}{c}\\
 & = \sqrt{\frac{k}{c}} e^{\frac{\pi}{\tilde z}+\frac{\tan^{-1}\hat z}{\hat z}} - \left( \frac{\betaM - \betam}{c} \right)\left( 1+e^{\frac{\pi}{\tilde z}} \right) > 0. \label{clocon1} \numberthis
\end{align*}
We will use the inequality in the above form later in Lemma \ref{clocon2lem} to obtain a final condition. However, we can further rewrite this to obtain the inequality as in the statement of the Lemma.
%\begin{align*}
%& \sqrt{kc} e^{\frac{\tan^{-1}\hat z}{\hat z}} > (\betaM-\betam)\left( 1 + e^{-\frac{\pi}{\tilde z}} \right). 
%\end{align*}
\end{proof}

\begin{remark}
\label{notsharp1}
The condition in Lemma \ref{clocon1lem} is sharp. A relaxed condition could be derived using Lemma \ref{p1astbound}, 
\begin{align*}
& \betaM - \betam < e^{\frac{\pi}{\tilde z}}\left( \betam + (e-2)\sqrt{kc} \right).
\end{align*}
However, we will make use of the sharp condition because it is evident from \eqref{clocon1} that if $\psi\equiv constant$, then there is no need for such a condition. 
\end{remark}
The second segment of the boundary of invariant region is,
\begin{align}
\label{secondseg}
& C_2 = \{ (p,q): (\tilde p(t),\tilde q(t)), t\in [t_2^{we} , 0)\}.
\end{align} 
%Also, this segment $(q=\eta_2(p))$ satisfies the IVP \eqref{eta2}.

{\bf Step 3:} For the third segment, we again use \eqref{muMaux} but with different initial conditions than the ones for the first segment. The third segment should start from the ending point of the second segment, i.e., $\hat p(0) = p_2^{we}$ and $\hat q(0) = 1/c$. On solving, we obtain
\begin{align*}
& \hat p(t) = \frac{\betaM}{c} + e^{-\frac{\betaM t}{2}}\left[\left(p_2^{we} - \frac{\betaM}{c}\right)\cos\hat \theta t + \left(\frac{p^{we}_2\betaM}{2 \hat \theta} - \frac{\betaM^2}{2c\hat \theta} \right)\sin\hat\theta t \right],\\
& \hat q(t) = \frac{1}{c} + \frac{e^{-\frac{\betaM t}{2}}}{\hat \theta}\left(p_2^{we} - \frac{\betaM}{c}\right)\sin \hat\theta t.
\end{align*}
Set 
\begin{align}
\label{p3ast}
& p_3^{we} := \hat p(t_3^{we}),
\end{align}
where $t_3^{we}$ is the first negative time when $\hat q(t_3^{we} ) = 0$. Hence, $t_3^{we}$ is the largest negative root of the following,
\begin{align*}
& e^{-\frac{\betaM t_3^{we}}{2}}\sin\hat\theta t_3^{we} = -\frac{\hat\theta}{cp_2^{we} - \betaM}.
%\theta_M e^{\frac{\mu_M t_3^\ast}{2}} + (cp_2^\ast - \mu_M) \sin\theta_M t_3^\ast = 0.
\end{align*}
To ensure the invariant region is closed, it should be that while traveling in the negative time direction, the trajectory hits the $p$-axis first before completing the outward spiral turn. The following Lemma ensures this.
\begin{lemma}
\label{clocon2lem}
Let $t_0$ be the first negative time such that $\hat q'(t_0) = 0$. Then $\hat q(t_0)<0$ if and only if 
\begin{equation}
\label{clocon2}
(\betaM-\betam) \left( 1 + e^{-\frac{\pi}{\tilde z}} \right) e^{- \frac{\tan^{-1}(\hat z)}{\hat z}} < \sqrt{kc}\left(1- e^{-\frac{\pi}{\hat z}-\frac{\pi}{\tilde z}}  \right).
\end{equation}
\end{lemma}
\begin{proof}
Solving for $\hat q(t) = 0$, we obtain that 
$$
\hat\theta t_0 = -\pi + \tan^{-1}(\hat z).
$$
Applying the condition $\hat q(t_0)<0$, we obtain 
\begin{align*}
& \frac{1}{c} + \frac{e^{-\frac{\betaM t_0}{2}}}{\hat \theta}\left(p_2^{we} - \frac{\betaM}{c}\right)\sin \hat\theta t_0 <0\\ 
& e^{-\frac{\betaM t_0}{2}}\sin\hat\theta t_0 < -\frac{\hat\theta}{cp_2^{we} - \betaM}.
\end{align*}
Plugging in the value of $t_0$, we need that
\begin{align*}
& e^{\frac{\pi}{\hat z} - \frac{\tan^{-1}(\hat z)}{\hat z}}(cp_2^{we} - \betaM) - \sqrt{kc} > 0.
\end{align*}
Note that \eqref{clocon1} in Lemma \ref{clocon1lem} is indeed a necessary and sufficient condition for $p_2^{we} - \betaM/c>0$ to hold. Hence, we can use \eqref{clocon1} in the above expression to obtain a single final condition. To this end, we want
\begin{align*}
& e^{\frac{\pi}{\hat z} - \frac{\tan^{-1}(\hat z)}{\hat z}}(cp_2^{we} - \betaM) - \sqrt{kc}\\
&\qquad=  e^{\frac{\pi}{\hat z} - \frac{\tan^{-1}(\hat z)}{\hat z}}\left( \sqrt{kc} e^{\frac{\pi}{\tilde z} + \frac{\tan^{-1}(\hat z)}{\hat z}} - (\betaM- \betam)(1+e^{\frac{\pi}{\tilde z}} ) \right) - \sqrt{kc}\\
&\qquad= \sqrt{kc}\left( e^{\frac{\pi}{\hat z}+\frac{\pi}{\tilde z}} - 1 \right) - (\betaM- \betam)(1+e^{\frac{\pi}{\tilde z}} ) e^{\frac{\pi}{\hat z} - \frac{\tan^{-1}(\hat z)}{\hat z}}>0.
\end{align*}
This finishes the proof to the Lemma.
\end{proof}
Finally, we can define the last segment of the boundary of $\Sigma_1^\ast$,
\begin{align}
\label{thirdseg}
& C_3 = \{ (p,q): (\hat p(t),\hat q(t)), t\in [t_3^\ast , 0)\}.
\end{align}
%Note that this segment $(q=\eta_3(p))$ satisfies the IVP \eqref{eta3}. 
We define the following set
\begin{align}
& \Sigma_1^\ast: = \text{open set enclosed by $C_1,C_2,C_3$ and $p$-axis}.    
\end{align}
By our construction, we know that $\Sigma_1^\ast$ is well-defined.

Next, we have the following Proposition.
\begin{proposition}
\label{propmainwal}
Let $4kc>\betaM^2$. Let the initial conditions for \eqref{mainodesys} be such that $(w(0),s(0))\in\Sigma_1^\ast$. Then $(w(t),s(t))\in\Sigma_1^\ast$ for all $t>0$.
\end{proposition}
We will prove the Proposition by drawing comparison between the solution trajectory $(w,s)$ and the boundary of $\Sigma_1^\ast$. Due to the presence of oscillations, a time based comparison between systems \eqref{mainodesys} and \eqref{mumMaux} cannot be derived. To circumvent this, we will draw comparisons in the $(p,q)$ plane.
\begin{proof} 
We will show that a solution trajectory to \eqref{mainodesys} with initial data in $\Sigma_1^\ast$ can never touch its boundary as time increases. 
By getting rid of the time parameter $t$ in the systems \eqref{muMaux} and \eqref{mumaux}, we obtain the following two trajectory equations below. These will play a significant role in proving the invariance of $\Sigma_1^\ast$.
\begin{subequations}
\label{etas}
\begin{align}
& \frac{d\hat q}{dp} = \frac{p - \hat q \betaM}{k-kc\hat q}, \label{eta1}\\
& \frac{d\tilde q}{dp} = \frac{p -\tilde q \betam}{k-kc\tilde q}. \label{eta2}
%& \frac{d\eta_3}{dp} = \frac{p - \mu_M\eta_3}{k-kc\eta_3}. \label{eta3}
\end{align}
\end{subequations}

We start by showing a contradiction if the trajectory touches $C_1$. To this end, assume a point $(w_1, s_1)\in C_1$ where the trajectory meets $C_1$. Therefore, $w_1\leq 0$ and $s_1<1/c$. For a reminder, any portion of $C_1$ is $(p,\hat q(p))$ with appropriate initial conditions and values of $p$. We also get rid of the time parameter in \eqref{mainodesys} to write $s$ as a function of another variable and satisfying,
\begin{equation}
\label{slesspara}
\frac{ds}{dp} = \frac{p-s\psi\ast\rho}{k-kcs}.
\end{equation}
We have $s_1 = \hat q(w_1)= s(w_1)$.
Since $\left.w'\right|_{(w_1,s_1)}>0$, the trajectory $(p,s(p))$ was moving in the positive $p$ direction before touching $C_1$, see Figure \ref{propfigC1}. Note that,
\begin{align*}
\frac{d(\hat q - s)}{dp} & = \frac{p- \hat q \betaM}{k-kc\hat q}- \frac{p-s(\psi\ast\rho) }{k-kcs} \\
& =  \frac{-cp(s - \hat q )-(  \betaM\hat q - (\psi\ast\rho) s)+ cs \hat q( \betaM - \psi\ast\rho)}{k(1-c\hat q)(1-c s )}\\ 
& = \frac{(cp - \psi\ast\rho)(\hat q-s) - \hat q(\betaM - \psi\ast\rho)(1-cs)}{k(1-c\hat q)(1-cs)}\\
& = \frac{(cp - \psi\ast\rho)}{k(1-c\hat q)(1-cs)}(\hat q -s) - \frac{\hat q(\betaM - \psi\ast\rho)}{k (1-c\hat q)}
\end{align*}
In a neighborhood of $p = w_1$ (if $w_1=0$ consider left neighborhood),
\[
  \frac{d(\hat q - s)}{dp}  \leq \frac{(cp- \psi\ast\rho) }{ k(1-c\hat q) (1-cs)} (\hat q -s).
\]
Upon integration in the interval $(w_1 - \epsilon,w_1)$, $\epsilon>0$ being sufficiently small, we obtain
\[
0 = \hat q(w_1) - s(w_1) \leq \left(\hat q (w_1-\epsilon)- s(w_1 - \epsilon) \right) e^{\int_{w_1 - \epsilon}^{w_1} \frac{(cp-\psi\ast\rho)}{k(1-c\hat q(p)) (1-cs(p))}dp} <0.
\]
This is a contradiction. Hence, a trajectory with initial point inside $\Sigma_1^\ast$ can never touch $C_1$. A very similar argument holds for $C_3$.

\begin{figure}[ht!] 
\centering
\subfigure[Trajectory touching $C_1$.]{\label{propfigC1}\includegraphics[width=0.4\textwidth]{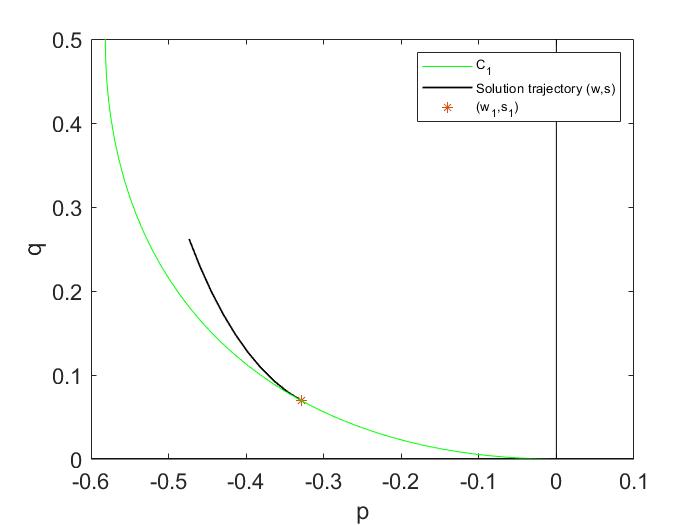}}
%\caption{Trajectory touching $C_1$.}
\qquad
\subfigure[Trajectory touching $C_2$.]{\label{propfigC2}\includegraphics[width=0.4\textwidth]{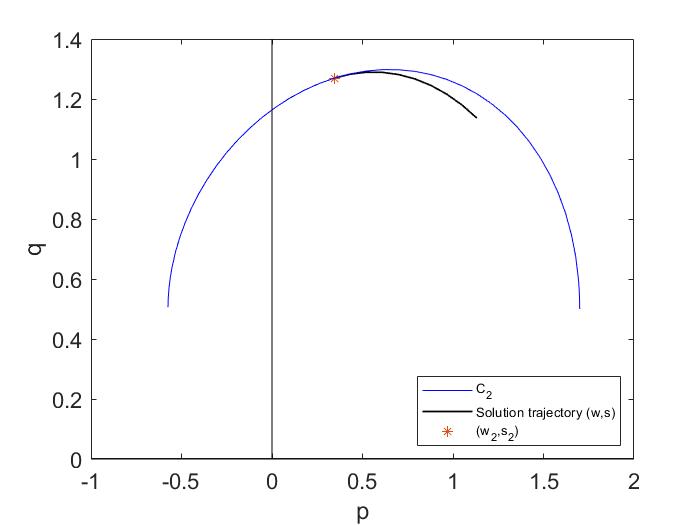}}
%\caption{Trajectory touching $C_2$.}
\caption{Trajectories touching boundary of $\Sigma_1^\ast$.}
\end{figure}

Now we show for $C_2$. For sake of contradiction, suppose there exists a point $(w_2,s_2)\in C_2$ where the trajectory, $(w,s)$, touches $C_2$. For a reminder, any portion of $C_2$ is $(p,\tilde q(p))$ with appropriate initial conditions and range of $p$. Owing to our assumptions, we have $1/c< s_2 = s(w_2) = \tilde q(w_2)$. Since $\left.w'\right|_{(w_2,s_2)}<0$, the solution trajectory $(p,s(p))$ was traveling in the negative $p$ direction when it touched $C_2$, see Figure \ref{propfigC2}. Similar to our previous calculations we obtain from \eqref{eta2} and \eqref{slesspara} that,
\begin{align*}
\frac{d(\tilde q- s)}{dp} & = \frac{(cp-\psi\ast\rho)}{k(1-c s)(1-c\tilde q)} (\tilde q -s) + \frac{\tilde q(\psi\ast\rho-\betam)}{k(1-c\tilde q)}.
\end{align*}
In a neighborhood of $p=w_2$,
\begin{align*}
\frac{d(\tilde q- s)}{dp} & \leq \frac{(cp-\psi\ast\rho)}{k(1-c s)(1-c\tilde q)} (\tilde q -  s) .
\end{align*}
Upon integration in the interval $(w_2, w_2 + \epsilon)$, for $\epsilon>0$ sufficiently small, we obtain
\begin{align*}
0<\tilde q(w_2 + \epsilon)- s(w_2 + \epsilon) & \leq \left( \tilde q(w_2) - s(w_2) \right) e^{\int_{w_2}^{w_2+\epsilon}\frac{cp-\psi\ast\rho}{k(1-cs(p))(1-c\tilde q(p))}dp} = 0.
\end{align*}
Hence, the solution trajectory cannot cross $C_2$.

Moreover, a trajectory $(w,s)$ starting from any point $(p,0)$ with $p>0$ will go up into the region because at any such point,
\begin{align*}
& \left.s'\right|_{(p,0)} = \left.w - (\psi\ast\rho)s \right|_{(p,0)} = p >0.
\end{align*}
This completes the proof to the proposition.
\end{proof}

Now we will transform $\Sigma_1^\ast$ to obtain an invariant region for \eqref{Grhosys}. To this end, define a map by $F:\mathbb{R}^2 \to \mathbb{R}^2$,
\begin{align}
\label{transmap1}
F(p,q) = (p/q,1/q).
\end{align}
$F$ is invertible for $q>0$. We define 
\begin{align}
\label{invregtransf}
\Sigma_1 := F(\Sigma_1^*),  
\end{align}
which is an invariant region for $(G, \rho)$. See Figure \ref{subcrGrhoweak} for the shape of the subcritical region, $\Sigma_1$, in $(G,\rho)$ coordinates.

\begin{figure}[ht!] 
\centering
\subfigure{\includegraphics[width=0.5\linewidth]{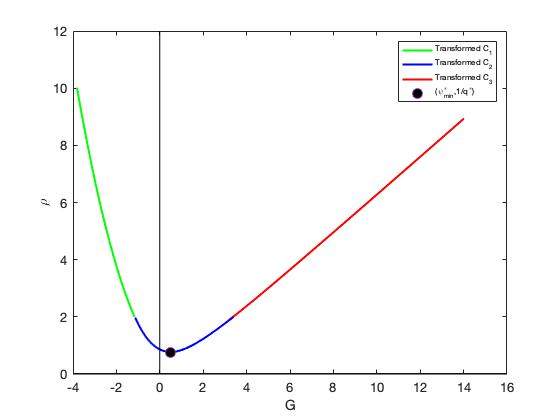}}
\caption{$C_1,C_2,C_3$ as transformed to original coordinates $(G,\rho)$.}
\label{subcrGrhoweak}
\end{figure}

\begin{remark}
\label{qastrem}
%We define the invariant region $\Sigma_1$ using the map \eqref{transmap1}. However, since $C_1,C_2,C_3$ are merely solution trajectories to a linear system, we can indeed denote $\Sigma_1$ through functions representing these solutions. In particular, there exist two Lipschitz continuous functions $\xi_1,\xi_2$ such that 
%\begin{align*}
%\Sigma_1 = & \left\{ (G,\rho):   \xi_1(\rho)< G <  \xi_2(\rho) ,\ \rho\in (1/q^\ast,\infty) \right\},
%\end{align*}
%where $q^\ast$ is the $q$ coordinate of the highest point of $C_2$ on the $(p,q)$ plane. $q^\ast$ can be explicitly calculated,
In Figure \ref{subcrGrhoweak}, the $\rho$ coordinate of the tip of $\Sigma_1$ is $1/q^\ast$, where $q^\ast$ is the $q$ coordinate of the highest point of $\Sigma_1^\ast$, see Figure \ref{fig1}. The expression of $q^\ast$ can be explicitly written as,
\begin{align}
\label{qastexp}
& q^\ast = \frac{1}{c} +  \frac{e^{\frac{\pi}{\tilde z} - \frac{\tan^{-1}(\tilde z)}{\tilde z}}}{\sqrt{kc}}\left( \frac{\betam}{c} - p_1^{we} \right).
\end{align}
Since $p_1^{we}<0$ from Lemma \ref{p1astbound}, we have that $q^\ast >1/c$.
Also, from Figure \ref{subcrGrhoweak}, we see that there is no point $(G,\rho)$ in $\Sigma_1$ such that $\rho < 1/q^\ast$.
\end{remark}
As a direct result of Proposition \ref{propmainwal} and transformation \eqref{vartransf}, we have the following corollary.
\begin{corollary}
\label{corwal}
Let $4kc>\betaM^2$ and \eqref{clocon2} holds. Let initial conditions for \eqref{Grhosys} be such that $(G(0),\rho(0))\in\Sigma_1$. Then $(G(t),\rho(t))\in\Sigma_1$ for all $t>0$. In particular, $G,\rho$ are bounded for any time.
\end{corollary}

\subsection{Strong alignment}
\label{strongal}
Now, we handle the case where all admissible values of $\beta\in[\betam, \betaM]$ in \eqref{auxsys} are such that 
$$
\beta^2\geq 4kc.
$$
In such a case $(\beta/c,1/c)$ is an asymptotically stable node and the solutions to \eqref{muMaux} and \eqref{mumaux} will not have any sinusoidal components. 
%Since there is no restriction on the upper bound of the alignment force, 
We call this scenario the strong alignment case. As before, we will construct an invariant region using specific trajectories. Unlike the invariant region constructed in Section \ref{weakal}, here we will have an unbounded subcritical region, $\Sigma_2^\ast$, see Figure \ref{fig2}. 

\begin{figure}[ht!] 
\centering
\subfigure{\includegraphics[width=0.5\linewidth]{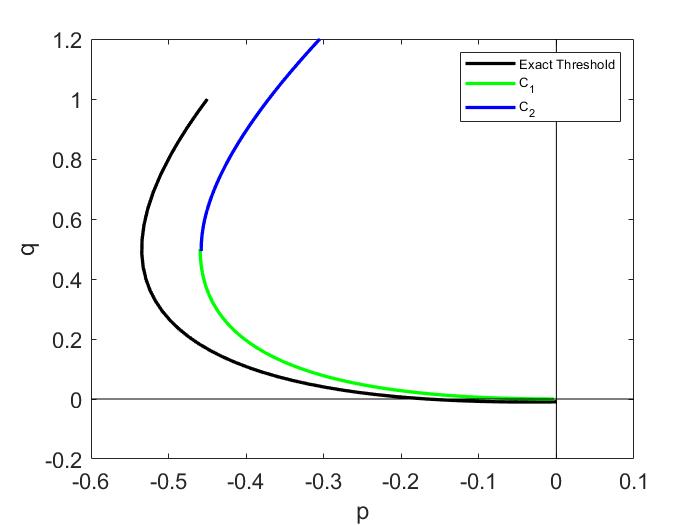}}
\caption{Invariant region.}
\label{fig2}
\end{figure}

%Set $\hat\gamma := \sqrt{\psi_{max}^{\ast 2} - 4kc}$ and $\tilde\gamma := \sqrt{\psi_{min}^{\ast 2} - 4kc}$.
We first define some notation to be used in construction of $\Sigma_2^\ast$. Set
\begin{align*}
& \hat\gamma_+ := \frac{\betaM+\sqrt{\betaM^2 - 4kc}}{2},\quad \hat\gamma_- := \frac{\betaM-\sqrt{\betaM^2- 4kc}}{2},\\
& \tilde\gamma_+ := \frac{\betam+\sqrt{\betam^2- 4kc}}{2},\quad \tilde\gamma_- := \frac{\betam-\sqrt{\betam^2 - 4kc}}{2}.
\end{align*}
\begin{remark}
\label{borderlineremark}
In this Section as well as Section \ref{mediumal}, we should point out that if $\betaM^2 = 4kc$, then the expressions of $\hat p,\hat q$ have different form than the ones when $\betaM^2 > 4kc$, which is assumed for calculations below. However, the calculated expressions for $p_1^{se},p_2^{se},p_3^{me}$ always hold, although in the limit sense when $\betaM^2 = 4kc$. A more detailed note about this is mentioned right after the proof of Lemma \ref{p1selem}.
\end{remark}

{\bf Step 1:} The first segment of the curve, $C_1$, is the trajectory to \eqref{muMaux} with the starting point at origin and the ending point lying on the line $q=1/c$ in the second quadrant. Set $p_1^{se} := \hat p(t_1^{se})$ so that $(p_1^{se},1/c)$ is the end point of $C_1$ lying in the second quadrant. Here, $t_1^{se}$ is the negative time when $\hat q(t_1^{se}) = 1/c$ with $\hat p(0) =\hat q(0)= 0$. On solving, we obtain,
\begin{align*}
& \hat p(t) = \frac{\betaM}{c} - \frac{k}{\sqrt{ \betaM^2- 4kc}}\left( \frac{\hat\gamma_+}{ \hat\gamma_-}e^{-\hat\gamma_- t} - \frac{\hat\gamma_-}{ \hat\gamma_+}e^{-\hat\gamma_+ t} \right),  \\
%& \hat p(t) = \frac{\psi_{max}^\ast}{c} - \frac{k(\psi_{max}^\ast+\hat\gamma)}{\hat\gamma(\psi_{max}^\ast-\hat\gamma)}e^{-(\psi_{max}^\ast - \hat\gamma)\frac{t}{2}} + \frac{k(\psi_{max}^\ast-\hat\gamma)}{\hat\gamma(\psi_{max}^\ast+\hat\gamma)}e^{-(\psi_{max}^\ast + \hat\gamma)\frac{t}{2}},\\
& \hat q(t) = \frac{1}{c} - \frac{1}{c\sqrt{ \betaM^2- 4kc}} \left( \hat\gamma_+ e^{-\hat\gamma_- t} - \hat\gamma_- e^{-\hat\gamma_+ t} \right).
%& \hat q(t) = \frac{1}{c} - \frac{(\psi_{max}^\ast + \hat\gamma)}{2c\hat\gamma}e^{-(\psi_{max}^\ast - \hat\gamma)\frac{t}{2}} + \frac{(\psi_{max}^\ast-\hat\gamma_)}{2c\hat\gamma} e^{-(\psi_{max}^\ast + \hat\gamma)\frac{t}{2}}.
\end{align*}
When $\hat q(t_1^{se}) = 1/c$,
\begin{align*}
& e^{\left(\sqrt{ \betaM^2- 4kc}\right) t_1^{se}} =  \frac{\hat\gamma_-}{\hat\gamma_+}.
\end{align*}
Consequently, using the fact that $\hat\gamma_+\hat\gamma_- = kc$, we have
\begin{align*}
p_1^{se} & = \frac{\betaM}{c} - \frac{k}{\sqrt{ \betaM^2- 4kc}}e^{-\frac{\betaM t_1^{se}}{2}}\left( \frac{\hat\gamma_+}{ \hat\gamma_-}e^{\frac{t_1^{se} }{2} \sqrt{ \betaM^2- 4kc}} - \frac{\hat\gamma_-}{ \hat\gamma_+}e^{-\frac{t_1^{se} }{2} \sqrt{ \betaM^2- 4kc}} \right) \\
& = \frac{\betaM}{c} + \frac{k}{\sqrt{\betaM^2- 4kc}}e^{-\frac{\betaM t_1^{se}}{2}}\left( -\sqrt{\frac{ \hat\gamma_+}{\hat\gamma_-}} + \sqrt{\frac{ \hat\gamma_-}{ \hat\gamma_+}} \right)\\
& = \frac{\betaM}{c} - k e^{-\frac{\betaM t_1^{se}}{2}}\left( \frac{1}{\sqrt{\hat\gamma_- \hat\gamma_+}} \right)\\
& = \frac{\betaM}{c} - \sqrt{\frac{k}{c}}\left( \frac{\hat\gamma_+}{ \hat\gamma_-} \right)^{\frac{\betaM}{2\sqrt{\betaM^2 - 4kc }}}. \label{p1ast2} \numberthis
\end{align*}
\begin{lemma}
\label{p1selem}
$p_1^{se}\in \sqrt{\frac{k}{c}}\,\left[-(e-2),0\right)$.
\end{lemma}
\begin{proof}
\eqref{p1ast2} can be rewritten as a function of essentially one variable,
$$
g(\tau) = \sqrt{\frac{k}{c}} \left( \frac{1}{\tau} - \left( \frac{1 + \sqrt{1 - 4\tau^2}}{1 - \sqrt{1 - 4\tau^2}} \right)^{\frac{1}{2\sqrt{1-4\tau^2}}}  \right), \quad \tau = \frac{\sqrt{kc}}{\betaM},\ \tau\in (0,1/2].
$$
One can check that the above function is decreasing with 
$$
\lim_{\tau\to(1/2)^-}g(\tau) = (2-e)\sqrt{\frac kc}\quad\text{and}\quad
\lim_{\tau\to 0^+} g(\tau) = 0.
$$ 
Hence, the result holds.
\end{proof}
%Compare this to the bounds of $p_1^{we}$ in Lemma \ref{p1astbound}. 
Since the ODE system \eqref{muMaux} is well-posed, $g(1/2) = f(1/2) = -(e-2)\sqrt{k/c} $, where $f$ is as defined in proof of Lemma \ref{p1astbound}. Moreover, if $\betaM = 2\sqrt{kc}$ (or equivalently $\tau = 1/2$), then the point $p_1^{se} = p_1^{we} =  -(e-2)\sqrt{k/c}$. The relation between $f$ and $g$ is much more. In fact, they are equal if we extend each of their domains to $\mathbb{R}^+$, see Remark \ref{p1selem2}.  
\begin{remark}
\label{p1selem2}
%$$
%f(\tau) = g(\tau),\quad \tau\in(0,\infty).
%$$
We recall $f$ here,
$$
f(\tau) = \sqrt{\frac{k}{c}}\left( \frac{1}{\tau} - e^{\frac{\tan^{-1}\sqrt{4\tau^2-1}}{\sqrt{4\tau^2 - 1}}} \right),\quad \tau\geq \frac{1}{2}.
$$
As a function into $\mathbb{R}$, $f$ is defined only for $\tau\geq 1/2$. We aim to extend it to accommodate $\tau\in\mathbb{R}^+$. It turns out that 
%Note that to prove that the two functions are equal, it suffices to show
$$
e^{\frac{\tan^{-1}\sqrt{4\tau^2-1}}{\sqrt{4\tau^2 - 1}}} = \left( \frac{1 + \sqrt{1 - 4\tau^2}}{1 - \sqrt{1 - 4\tau^2}} \right)^{\frac{1}{2\sqrt{1-4\tau^2}}} ,\quad \tau \in(0,\infty).
$$
To see this, let $z := \sqrt{4\tau^2-1}, y:= \sqrt{1-4\tau^2}$. Consequently, $e^{\frac{\tan^{-1}z}{z}} = e^{\frac{\tan^{-1}iy}{iy}}=:h(y)$. We have,
\begin{align*}
%& \ln h = \frac{\tan^{-1}iy}{iy}\\
& \tan(\ln (h^{iy})) = iy\\
& \frac{e^{i\ln(h^{iy})} - e^{-i\ln(h^{iy})}}{i\left( e^{i\ln(h^{iy})} + e^{-i\ln(h^{iy})} \right)} = iy\\
& \frac{h^{-y} - h^y}{h^{-y}+h^y} = -y\\
& h^{2y} = \frac{1+y}{1-y},
\end{align*}
and finally,
\begin{align*}
& h(y) = \left( \frac{1+y}{1-y} \right)^{\frac{1}{2y}}.
\end{align*}
\end{remark}
Owing to Remark \ref{p1selem2}, the formula for $p_1^{se}$ is the same as $p_1^{we}$, which is,
$$
p_1^{se} = \frac{\betaM}{c} - \sqrt{\frac{k}{c}} e^{\frac{\tan^{-1}(\hat z)}{\hat z}},
$$
where $\hat z$ is purely imaginary and output of $\tan^{-1}$ is the principal value.

We now define the first segment of boundary of $\Sigma_2^\ast$,
\begin{align}
\label{salc1}
& C_1 = \{ (p,q): (\hat p(t),\hat q(t)),\ t\in [t_1^{se} ,0)  \}.
\end{align}

%Similar to \eqref{invboundeqsa}, the first segment of the curve is given by $( \hat\xi_-( q_-),q_-)$, $q_-\in (0,1/c)$ where $\hat\xi_-, q_-$ is as in \eqref{invboundeqsa}.
%\begin{align}
%& \frac{d\hat\xi_-}{dq_-} = -\frac{kc q_- }{\hat\xi_-+\psi_{max}^\ast q_- - %\frac{\psi_{max}^\ast}{c}},\quad \hat\xi_-(0) = p_1^{se},\ q_- = 1/c - q.
%\end{align}

{\bf Step 2:} Now, we move on to the second segment. For this part, we need the solutions to \eqref{mumaux} with initial condition $\tilde p(0) = p_1^{se}$ and $\tilde q(0) = 1/c$. Hence,
\begin{align*}
%& \tilde p (t) = \frac{\psi_{min}^\ast}{c} + \frac{(\psi_{min}^\ast - cp_1^{se})}{2c\tilde\gamma} e^{-\frac{\psi_{min}^\ast t}{2}} \left( (\psi_{min}^\ast - \tilde\gamma)e^{-\frac{\tilde\gamma t}{2}} - (\psi_{min}^\ast + \tilde\gamma)e^{\frac{\tilde\gamma t}{2}} \right),\\
& \tilde p (t) = \frac{\betam}{c} + \frac{(\betam - cp_1^{se})}{c\sqrt{\betam^2-4kc }} \left( \tilde\gamma_- e^{-\tilde\gamma_+ t} - \tilde\gamma_+ e^{-\tilde\gamma_- t} \right), \\
%& \tilde q(t) = \frac{1}{c} + \frac{(\psi_{min}^\ast - cp_1^{se})}{\tilde\gamma}e^{-\frac{ \psi_{min}^\ast t}{2}}\left ( e^{-\frac{\tilde\gamma t}{2}} - e^{\frac{\tilde\gamma t}{2}} \right).\\
& \tilde q(t) = \frac{1}{c} + \frac{(\betam - cp_1^{se})}{\sqrt{\betam^2-4kc }} \left( e^{-\tilde\gamma_+ t} - e^{-\tilde\gamma_- t} \right). \\
\end{align*}
Note that $\tilde p,\tilde q$ are strictly decreasing for $t<0$ and $\lim_{t\to-\infty}  \tilde p(t) = \lim_{t\to-\infty} \tilde q(t) = \infty $.
We now define $C_2$.
\begin{align}
\label{salc2}
& C_2 = \{ (p,q): (\tilde p(t),\tilde q(t)),\ t\in (-\infty , 0] \}.
\end{align}
%Consequently, the second segment of the boundary of invariant region is given by $(\tilde\xi(q),q)$, $q\geq 1/c$ where $\tilde\xi$ is as in \eqref{invboundeqsb}. 
%\begin{align}
%\label{salxi2}
%& \frac{d\tilde\xi}{dq} = \frac{k-kcq}{\tilde\xi - \psi_{min}^\ast q},\quad \tilde\xi(1/c) = %p_1^{se}.
%\end{align}
This completes our construction and we are ready to define $\Sigma_2^\ast$.
%Define the region
\begin{align}
& \Sigma_2^\ast = \text{unbounded open set surrounded by } C_1, C_2, \{(p,0): p>0\} \text{ on 3 sides} .
\end{align}
Our construction ensures $\Sigma_2^\ast$ is well-defined. The following proposition states the invariance of $\Sigma_2^\ast$.
\begin{proposition}
\label{propmainsal}
Let $4kc\leq\betam^2$. Consider the ODE system \eqref{mainodesys}. If $(w(0),s(0))\in \Sigma_2^\ast$ then $(w(t),s(t))\in\Sigma_2^\ast$ for all $t>0$. In particular, $w,s$ remain bounded and $s(t)>0$ for all time.
\end{proposition}
\begin{proof}
The proof for the part that $(w,s)$ never crosses $C_1$ or $C_2$ is very similar to that in the proof of Proposition \ref{propmainwal}. So, we will omit it here. We prove that $s>0$ and $w,s$ remain bounded.

The only points where the trajectory $(w,s)$ could cross the $p$-axis are of the form $(p,0)$ where $p>0$. However, at any such point, $s'>0$ and therefore, the trajectory moves upwards. Consequently, $s(t)>0$ for all $t>0$. As a result,
\begin{align*}
w' & = k - kcs < k. 
\end{align*}
Therefore, $w$ is bounded from above. Moreover,
\begin{align*}
& s' = w - s\psi\ast\rho \leq w,
\end{align*}
and hence, $s$ is bounded from above.
\end{proof}
Similar to what we did in Section \ref{weakal}, we will now transform $\Sigma_2^\ast$ to obtain an invariant region for \eqref{Grhosys}. However, due to the fact that $\Sigma_2^\ast$ is unbounded, through \eqref{vartransf} we have that there are points in $F(\Sigma_2^\ast)$ with positive but arbitrarily small values of $\rho$. This indicates that the subcritical region might contain points where $\rho=0$ which we miss in the above analysis due to working with the transformed variables, \eqref{vartransf}. Indeed if $\rho = 0$ in \eqref{Geq}, then
\begin{align*}
G' & = -G(G-\psi\ast\rho ) - kc = -(G^2 -G\psi\ast\rho + kc)\\
& = -\left(G-\frac{\psi\ast\rho - \sqrt{(\psi\ast\rho)^2-4kc}}{2}\right) \left(G-\frac{\psi\ast\rho + \sqrt{(\psi\ast\rho)^2-4kc}}{2}\right).  \label{Groots} \numberthis
\end{align*}
Noting that 
\begin{align*}
& \max \left(\psi\ast\rho - \sqrt{(\psi\ast\rho)^2-4kc}\right) = \betam-\sqrt{\betam^2-4kc}\\ 
& \quad <\betam+\sqrt{\betam^2-4kc}  = \min \left(\psi\ast\rho + \sqrt{(\psi\ast\rho)^2-4kc}\right),
\end{align*}
therefore, if 
\begin{align}
\label{gsrhois0}
G(0)>  \frac{\betam-\sqrt{\betam^2-4kc}}{2},
\end{align}
then $G(t)$ is bounded for all times. So, due to the balancing effect of the strong alignment, we have subcritical data for $\rho=0$ as well, which was not the case for $\Sigma_1$ in Section \ref{weakal}.
Owing to the above analysis and using $F$ as in \eqref{transmap1}, we define 
\begin{align}
\label{invregtransf2}
\Sigma_2 := F(\Sigma_2^*) \cup \left\{ (G,0): G> \frac{\betam - \sqrt{\betam^2 - 4kc}}{2}  \right\},  
\end{align}
which is an invariant region for $(G, \rho)$. See Figure \ref{subcrGrhostrong} for the shape of the subcritical region in $(G,\rho)$ coordinates.
\begin{figure}[ht!] 
\centering
\subfigure{\includegraphics[width=0.7\linewidth]{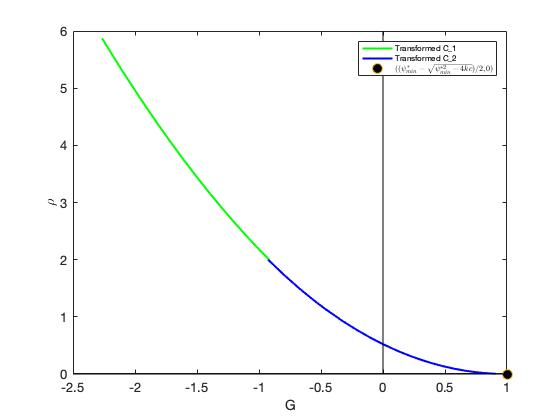}}
\caption{$C_1,C_2$ as transformed to original coordinates $(G,\rho)$.}
\label{subcrGrhostrong}
\end{figure}

\begin{remark}
We define the invariant region $\Sigma_2$ using the map \eqref{transmap1}. However, since $C_1,C_2$ are merely solution trajectories to a linear system, we can indeed denote $\Sigma_2$ through a function representing these solutions. In particular, there exists a Lipschitz continuous function $\xi_a$ such that 
\begin{align*}
\Sigma_2 = & \left\{ (G,\rho):  G> \xi_a(\rho) ,\ \rho\in [0,\infty) \right\}.
\end{align*}
\end{remark}

\iffalse
Transforming $\Sigma_2^\ast$ back to $G,\rho$, as we did in Lemma \ref{grhoinvreg}, we obtain the set,
\begin{align}
\label{invregtransf2}
\Sigma_2(p_1^{se})  = & \{(G,\rho): G> \rho\hat\xi_- (1/c - 1/\rho),\ \rho\in [c,\infty)\}\\
& \cup \{ (G,\rho): G>\rho\tilde\xi(1/\rho ),\ \rho\in [0,c]\}, \nonumber
\end{align}
where $p_1^{se}$ is defined in \eqref{p1ast2}.
\fi

\begin{proposition}
\label{corsal}
Let $4kc<\betam^2$. Let initial conditions for \eqref{Grhosys} be such that $(G(0),\rho(0))\in\Sigma_2$. Then $(G(t),\rho(t))\in\Sigma_2$ for all $t>0$. In particular, $G(t),\rho(t)$ are bounded for any time.
\end{proposition}
\begin{proof}
Note that if $\rho(0)=0$ in \eqref{Grhosys}, then $\rho\equiv 0$. 
Also, if $(G(0),\rho(0))\in F(\Sigma_2^\ast)$, then as a direct result of Proposition \ref{propmainsal} and transformations \eqref{vartransf}, we conclude that $(G(t),\rho(t))\in F(\Sigma_2^\ast)$ for all $t>0$. Consequently, $\rho(0)>0\implies \rho(t)>0$ for further times. In particular, this justifies that we can handle the $\rho(0)=0\equiv \rho$ case separately.
\iffalse
W $\rho=0$ as well. It has to be taken in the sense of $\lim_{\rho\to 0}\rho\tilde\xi(1/\rho)$. Set
\begin{align*}
\lim_{\rho\to 0}\rho\tilde\xi(1/\rho) & = \lim_{q\to\infty} \frac{\tilde\xi(q)}{q} =: L.
\end{align*}
This limit exists. To see this, we use L'H\^opital's rule on \eqref{invboundeqsb}. From phase plane analysis (or from expression of $\tilde q$), we know that $\lim_{q\to\infty}\tilde\xi(q) = \infty$. Using L'H\^opital's, 
\begin{align*}
& L = -\frac{kc}{L - \psi_{min}^\ast}\\
& L^2 - \psi_{min}^\ast L + kc = 0.
\end{align*}
From phase plane analysis of the system \eqref{mumaux}, we know that the two roots of the above equation are the slope of the eigenvectors to the two corresponding eigenvalues to the coefficient matrix. This is because the slope of all solution trajectories become infinitesimally close to that of one of the two eigenvectors as $t\to-\infty$. However, since the curve $(\tilde\xi,q)$ starts from $(p_1^{se} ,1/c)$, which is in the second quadrant, we have that $L$ is the smaller of the two roots owing to uniqueness of ODE solutions. Therefore, $L = (\psi_{min}^\ast - \sqrt{\psi_{min}^{\ast 2} - 4kc})/2$. 
\fi
From \eqref{gsrhois0} above, we conclude the result for this case. This finishes the proof to the Proposition.
\end{proof}

\subsection{Medium alignment}
\label{mediumal}
This is the case where the range of $\beta\in [\betam,\betaM]$ in \eqref{auxsys} is such that $\betam^2 < 4kc \leq \betaM^2$. We call this scenario the medium alignment case. Here, we will use analysis of both Sections \ref{weakal} and \ref{strongal}. The invariant region here is closed as in Section \ref{weakal} where $4kc>\betaM^2$. 
%The IVP equations for the boundary of the invariant region are same as \eqref{invboundeqs} with different values of initial data. 
The procedure to calculate $p_1^{me},p_2^{me},p_3^{me}$ is very
similar to what it is in Sections \ref{weakal} and \ref{strongal}. So,
we omit the calculations. We get
\[
  p_1^{me} = p_1^{se} = \frac{\betaM}{c} - \sqrt{\frac{k}{c}}e^{\frac{\tan^{-1}\hat z}{\hat z}},
\]
%where $\hat z$ is purely imaginary and output of $\tan^{-1}$ is the principal value. In particular, if $\psi_{max}^{\ast 2} = 4kc$, then 
%$$
%p_1^{me} =\sqrt{k/c}\left( 2 -  \lim_{\hat z\to 0} e^{\frac{\tan^{-1}\hat z}{\hat z}}\right) = -(e-2)\sqrt{(k/c)}.
%$$
\[
  C_1 = \left\{ (p,q): (\hat p(t),\hat q(t)),\ t\in [t_1^{se},0) \right\},
\]
with $\hat p,\hat q$ and $t_1^{se}$ as in Step 1 of Section
\ref{strongal}, and
\[
  p_2^{me} = \frac{\betam}{c}\left( 1+ e^{\frac{\pi}{\tilde z}} \right) - p_1^{me} e^{\frac{\pi}{\tilde z}}.
\]
Similar to the condition \eqref{clocon1} in Section \ref{weakal}, to have a closed invariant region, we need
$$
p_2^{me} > \frac{\betaM}{c}.
$$
We have the following Lemma. 
\begin{lemma}
\label{cloconmallem}
$p_2^{we}>\betaM/c$ if and only if 
\begin{align}
\label{clocon21}
& \sqrt{kc}e^{\frac{\tan^{-1}\hat z}{\hat z}} > (\betaM-\betam) \left( 1+ e^{-\frac{\pi}{\tilde z}} \right) .
\end{align}
\end{lemma}
The proof is very similar to that of Lemma \ref{clocon1lem}. We then have,
\begin{align*}
& C_2 = \left\{ (p,q): (\tilde p(t),\tilde q(t)),\ t\in [t_2^{we},0) \right\},
\end{align*}
with $\tilde p,\tilde q$ and $t_2^{we}$ as in Step 2 of Section \ref{weakal}.

We complete our construction by finding the point $p_3^{me}$ and the third segment of the boundary of the invariant region. The desired curve is the portion of the solution to \eqref{muMaux} with $\hat p(0) = p_2^{me}$ and $\hat q(0) = 1/c$. Using this, we obtain
\begin{align*}
& \hat p(t) = \frac{\betaM}{c} + \frac{\left(p_2^{me} - \frac{\betaM}{c}\right)}{\sqrt{\betaM^2 -4kc} } \left( \hat\gamma_+ e^{-\hat\gamma_- t} - \hat\gamma_- e^{-\hat\gamma_+ t} \right), \\
& \hat q(t) = \frac{1}{c} + \frac{\left( p_2^{me} - \frac{\betaM}{c}\right)}{\sqrt{\betaM^2 - 4kc }} \left( e^{-\hat\gamma_- t} - e^{-\hat\gamma_+ t} \right).
\end{align*}
$\hat q$ is unbounded and strictly increasing for $t<0$, hence, there exists a unique $t_3^{me}<0$ such that $\hat q(t^{me}_3 ) = 0$. 
%Consequently, here we do not need a \eqref{clocon2} type condition. 
$t_3^{me}$ is the unique solution of 
$$
e^{-\hat\gamma_+ t } - e^{-\hat\gamma_- t } = \frac{\hat\gamma}{cp_2^{me} - \betaM}.
$$
And $p_3^{me} := \hat p(t^{me}_3)$. Also,
$$
C_3 = \left\{ (p,q): (\hat p(t),\hat q(t)),\ t\in [t_3^{me} ,0) \right\}.
$$
And
\begin{align}
& \Sigma_3^\ast: = \text{open set enclosed by $C_1,C_2,C_3$ and $p$-axis}.    
\end{align}
Finally, we obtain the invariant region for the $(G,\rho)$ plane. Set
\begin{align}
\label{invregtransf3}
& \Sigma_3 = F(\Sigma_3^\ast).
\end{align}
We can now have the following Proposition.
\begin{proposition}
\label{cormal}
Let $\betam^2 < 4kc \leq \betaM$ and \eqref{clocon21} holds. Let initial conditions for \eqref{Grhosys} be such that $(G(0),\rho(0))\in\Sigma_3$. Then $(G(t),\rho(t))\in\Sigma_3$ for all $t>0$.
\end{proposition}
The proof is very similar to that of Proposition \ref{corwal}.

\subsection{Global smooth solutions}
We are now ready to prove Theorem \ref{gs1}.
\begin{proof}[Proof of Theorem \ref{gs1}]
Assume the hypothesis of Assertion (1). $\psiM<\g$ implies
$4kc>\betaM^2$ which means this case lies in the purview of Section
\ref{weakal}. Also, the admissible condition \eqref{eq:weakacc}
can be rewritten as \eqref{clocon2}. As a result, by Lemma
\ref{clocon2lem}, the invariant region $\Sigma_1$ is well defined.

Now suppose $(G_0(x),\rho_0(x))\in \Sigma_1$ for all
$x\in\mathbb{T}$. Then along any characteristic path \eqref{chpath},
the initial data to \eqref{Grhosys}, $(G(0),\rho(0))\in\Sigma_1$. By
Corollary \ref{corwal}, $(G(t),\rho(t))\in\Sigma_1$ for all $t>0$ and
$G(t),\rho(t)$ are bounded along all characteristic paths.
Finally, we can apply Theorem \ref{local} and conclude that $(\rho,u)$
are global-in-time smooth solutions to \eqref{mainsys}.

The proof to Assertions (2) and (3) in the Theorem is very similar only that in place of Corollary \ref{corwal} used above, we use Propositions \ref{corsal} and \ref{cormal} respectively.
\end{proof}

\subsection{Finite time breakdown}
This section is devoted to the proof of Theorem \ref{ftb1}. The procedure of construction is very similar to that in Sections \ref{weakal}, \ref{strongal} and \ref{mediumal}. The only difference is that we use the system \eqref{mumaux} wherever we used \eqref{muMaux} and vice-versa. As a result, $\betaM$ and $\betam$ interchange places in the relevant expressions. We only state the crucial steps and Propositions in obtaining the supercritical region.

\subsection*{Weak alignment ($4kc>\betaM^2$)}
We have 
\begin{align*}
%\label{ftbwalc1}
& B_1 = \left\{ (p,q): (\tilde p(t),\tilde q(t)),\ t\in [t_1^{we},0) \right\},
\end{align*}
where $\tilde p,\tilde q$ are solutions to \eqref{mumaux} with initial conditions $\tilde p(0) = \tilde q(0) = 0$, and $t_1^{we}<0$ is the first negative time when $\tilde q(t_1^{we} )=1/c$. Also,
$$
p_1^{we}:=\tilde p(t_1^{we}) = \frac{\betam}{c} - \sqrt{\frac{k}{c}}e^{\frac{\tan^{-1}\tilde z }{\tilde z}}.
$$
We have the same bounds of $p_1^{we}$ as in Lemma \ref{p1astbound}, $p_1^{we}\in \sqrt{k/c}\,(-1,-(2-e))$.
Next, we have
\begin{align*}
& B_2 = \left\{ (p,q): (\hat p(t),\hat q(t)),\ t\in [t_2^{we},0] \right\},
\end{align*}
where $\hat p,\hat q$ are solutions to \eqref{muMaux} with initial conditions $\hat p(0) = p_1^{we}, \hat q(0) = 1/c$, and $t_2^{we}<0$ is the first negative time when $\hat q(t_2^{we}) = 1/c$. Also,
$$
p_2^{we}:=\hat p(t_2^{we}) = \frac{\betaM}{c}\left( 1+ e^{\frac{\pi}{\hat z}} \right) - p_1^{we} e^{\frac{\pi}{\hat z}}.
$$
Since $p_1^{we}<0$, we have that $p_2^{we}>\betam/c$ and we do not need any extra condition (like \eqref{clocon1}) to close the invariant region. Lastly, 
$$
B_3 = \left\{ (p,q): (\tilde p(t),\tilde q(t)),\ t\in [t_3^{we},0) \right\},
$$
where $\tilde p,\tilde q$ are solutions to \eqref{mumaux} with initial conditions $\tilde p(0) = p_2^{we}, \tilde q(0) = 1/c$, and $t_3^{we}<0$ is the first negative time when $\tilde q(t_3^{we}) = 0$. Here again, we do not need any extra condition (like \eqref{clocon2}) for invariant region to be well-defined. To see this, just interchange $\betaM$ and $\betam$ to see that the right hand side of \eqref{clocon2} becomes negative. Therefore, the condition holds trivially. Finally, we define
\begin{align*}
& \Delta_1^\ast = \text{unbounded open set outside $B_1\cup B_2\cup B_3$ with } q>0.
\end{align*}
We then have the following Proposition.
\begin{proposition}
\label{ftbpropmainwal}
Let $4kc>\betaM^2$. Let the initial conditions for
\eqref{mainodesys} be such that $(w(0),s(0))\in\Delta_1^\ast$. Then
there exists $ t_c>0$ such that $s(t_c) = 0$. Also, $w(t_c)<0$.
\end{proposition}
\begin{proof}
The proof that the trajectory $(w,s)$ does not touch $\Delta_1^\ast$ is very similar to that of Proposition \ref{propmainwal}. And from the signs of $w',s'$, it can be concluded that the trajectory hits $q=0$ line for some time, $t_c>0$, in the second quadrant. Hence, $w(t_c)<0$. 
\end{proof}
Due to $\Delta_1^\ast$ being unbounded, we have supercritical region for points where $\rho=0$ as well. Here, we have all such points in the supercritical region.
\begin{equation}
\label{ftbinvregtransf}
\Delta_1 = F(\Delta_1^\ast)\cup \left\{ (G,0): G\in\mathbb{R} \right\}.
\end{equation}
To see the inclusion of the points $(G,0)$, we prove a Lemma. 
\begin{lemma}
\label{rho0points}
Let a function $h$ satisfy
$$
h' = -(h^2 -a(t)h + \omega),
$$
where $a$ is a bounded function and $\omega$ is a constant with $ 4\omega - \sup a^2>0 $. Then for any initial data $h(0)$, there exists $t_c>0$ such that $\lim_{t\to t_c^-} h(t) = -\infty$.
\end{lemma}
\begin{proof}
First, observe that %we show that $h$ becomes negative in some time if it is not already so. For the sake of contradiction assume $h(t)> 0$ for all $t>0$. Note that $h$ is bounded from above and therefore, well-defined for all time under our assumption. To see this,
\begin{align*}
h' & = -h^2 + a h - \omega\\
 & = -\left( h-\frac{a}{2}\right)^2 - \frac{4\omega-a^2}{4}\\
 &\leq - \frac{4\omega-\sup a^2}{4}<0.
\end{align*}
Therefore, $h$ is strictly decreasing and can achieve any negative number. In particular,  
%becomes negative in some time and remains 
for some $t_0\geq 0$, $h(t_0)<\min\{0,\inf a\}$. Consequently for $t>t_0$,
\begin{align*}
h' & < -h( h - \inf a).
\end{align*}
Since $h(t_0)<\min\{0,\inf a\}$, it admits a Riccati type blowup. Indeed on comparing above differential inequality with an equality, we obtain $\lim_{t\to t_c^-} h(t) = -\infty$ for some $t_c < t_0 + (-h(t_0))^{-1}$.
\end{proof}
%Keeping this in mind we define $\Delta_1$. Mapping through $F$ as in %\eqref{transmap1},
%\begin{align}
%\label{ftbinvregtransf}
%\Delta_1 = F(\Delta_1^\ast)\cup .
%\end{align}

Owing to this Lemma, Proposition \ref{ftbpropmainwal} and transformation \eqref{vartransf}, we have the following Corollary.
\begin{corollary}
\label{ftbcorwal}
Let $4kc>\betaM^2$. Let initial conditions for \eqref{Grhosys} be such
that $(G(0),\rho(0))\in\Delta_1$. Then there exists $ t_c >0$ such that 
$$
\lim_{t\to t_c^-} G (t,x_c)  = -\infty,\quad 
 \lim_{t\to t_c^-} \rho (t,x_c) = \infty\ \text{or } 0,
$$
for some $x_c\in \mathbb{T}$.
\end{corollary}
\begin{proof}
Note that if $\rho(0)=0$ in \eqref{Grhosys}, then $\rho\equiv 0$ and from \eqref{Geq},
\[
G' = -(G^2 - G\psi\ast\rho +kc).
\]
We can have $h=G, a(t)=\psi\ast\rho, \omega=kc$ in Lemma \ref{rho0points}  and the hypothesis is satisfied. Hence, for some $t_c>0$, $\lim_{t\to t_c^-}G(t) = -\infty$ irrespective of $G(0)$.
Also, if $\rho(0)>0$ with $(G(0),\rho(0))\in F(\Delta_1^\ast)$, then from Proposition \ref{ftbpropmainwal} and transformations \eqref{vartransf}, we have the existence of $t_c>0$ such that $\lim_{t\to t_c^-} \rho(t) = \infty$ and $\lim_{t\to t_c^-} G(t) = -\infty$. This finishes the proof to the Corollary.
\end{proof}

\subsection{Strong and medium alignment ($4kc\le\betaM^2$)}
These two cases are similar, so we state the construction together. We have
$$
B_1 = \left\{ (p,q): (\tilde p(t),\tilde q(t)),\ t\in [t_1^{se},0) \right\},
$$
where $\tilde p,\tilde q$ are solutions to \eqref{mumaux} with initial conditions $\tilde p(0) = \tilde q(0) = 0$, and $t_1^{se}<0$ is the first negative time (for medium alignment) and the unique time (for strong alignment), when $\tilde q(t_1^{se} )=1/c$.
Next,
$$
B_2 = \left\{ (p,q): (\hat p(t),\hat q(t)),\ t\in (-\infty,0] \right\},
$$
where $\hat p,\hat q$ are solutions to \eqref{muMaux} with initial conditions $\hat p(0) = p_1^{we}, \hat q(0) = 1/c$.
We can now define $\Delta_2^\ast$.
$$
\Delta_2^\ast = \text{unbounded open set surrounded by } B_1, B_2, \{(p,0): p<0\} \text{ on 3 sides} .
$$
%From a result similar to \eqref{gsrhois0}, due to $\Delta_2^\ast$ being unbounded, we have supercritical region for points where $\rho=0$ as well. 
Here , we have
\begin{equation}
\label{ftbinvregtransf2}
\Delta_2 = F(\Delta_2^\ast)\cup \left\{ (G,0): G< \frac{ \betaM - \sqrt{\betaM^2- 4kc }}{2} \right\}.
\end{equation}
Indeed when $\rho=0$, from \eqref{Groots} we have that if 
\[
  G(0) < \min \frac{\psi\ast\rho-\sqrt{(\psi\ast\rho)^2-4kc}}{2}
  = \frac{\betaM-\sqrt{\betaM^2-4kc}}{2},
\]
then $G(t)\to-\infty$ in finite time.

We have the following Proposition.
\begin{proposition}
\label{ftbcorsal}
Let $4kc\leq\betaM^2$. Let initial conditions for \eqref{Grhosys} be
such that $(G(0),\rho(0))\in\Delta_2$. Then there exists $t_c >0$ such that 
$$
\lim_{t\to t_c^-} G (t,x_c)  = -\infty,\quad 
 \lim_{t\to t_c^-} \rho (t,x_c) = \infty\ \text{or } 0,
$$
for some $x_c\in \mathbb{T}$.
\end{proposition}
We now give the proof to Theorem \ref{ftb1}.
\begin{proof}[Proof of Theorem \ref{ftb1}]
Assume the hypothesis of Assertion (1). $\psiM<\g$ implies
$4kc>\betaM^2$ which means this case lies in the purview of Section
\ref{weakal}. Now suppose $(G_0(x_0),\rho_0(x_0))\in \Delta_1$ for
some $x_0\in\mathbb{T}$. Then consider the dynamics \eqref{Grhosys}
along the characteristic path \eqref{chpath} with $x(0)=x_0$,
and apply Corollary \ref{ftbcorwal}. We have
$$
\lim_{t\to t_c^-} G (t,x_c) = -\infty,\quad  \lim_{t\to t_c^-} \rho (t,x_c)  = \infty\ \text{or } 0.
$$
This proves Assertion (1).
The proof to Assertion (2) in the Theorem is very similar only that in place of Corollary \ref{ftbcorwal} used above, we use Proposition \ref{ftbcorsal}.
\end{proof}

\section{The EPA system with weakly singular alignment influence}
\label{weaklysing}
In this section, we tackle the case when $\psi\in L^1_+(\mathbb{T})$. In
particular, $\psi$ need not be bounded as was assumed in Section
\ref{analysis1}. This type of alignment forces is known as
\emph{weakly singular}.

\subsection{Improved bounds on $\psi\ast\rho$}
The main difficulty of applying our theory in Section
\ref{analysis1} to the EPA system with weakly singular alignment
influence is that the bounds on $\psi\ast\rho$ in
\eqref{eq:psibounded} no longer hold.
A natural replacement of the bounds is \eqref{eq:psiweaksing}, which
we recall here:
\begin{equation}
\label{psistarcrudebounds}
\|\psi\|_{L^1}\rhom\leq\psi\ast\rho\leq\|\psi\|_{L^1}\rhoM.
\end{equation}
A major issue arises that the bounds depend on the unknown $\rho$. If
we were to pick $\rhoM$ and $\rhom$ and use the bounds
\eqref{psistarcrudebounds} in place of
$\betaM,\betam$ as in Section \ref{analysis1}, then 
the invariant region $\Sigma$ need to satisfy
\begin{equation}\label{eq:SigmaL}
\inf\{\rho: (G,\rho)\in\Sigma \}\geq \rhom, \quad\sup\{\rho: (G,\rho)\in \Sigma\}\leq \rhoM,
\end{equation}
in order to keep \eqref{psistarcrudebounds} valid.
However, after detailed analysis, it turns out that
there are no  
values of $\rhoM$ and $\rhom$ with which the constructed invariant
region $\Sigma$ satisfies \eqref{eq:SigmaL}.

To overcome this difficulty, we make improvements to the bounds
\eqref{psistarcrudebounds} 
leveraging the additional property on $\rho$
\[\int_{\mathbb{T}}\rho(t,x)\,dx=c,\quad\forall~t\ge0.\]

In particular, we have the following key Lemma.
\begin{lemma}
  \label{psistarnewboundslem}
  Let $\rho$ be any nonnegative, periodic function satisfying
  \begin{equation}\label{eq:constraint}
    \int_{\mathbb{T}}\rho(x)\,dx=c,\quad\text{and}\quad
    \rhom\le\rho(x)\le\rhoM,\,\,\forall~x\in\mathbb{T}.
  \end{equation}
Let $\psi\in L^1_+(\mathbb{T})$. Then there exist two non-negative
constants $\gamma_1$ and $\gamma_2$ such that
\begin{equation}\label{psistarnewbounds}
  \rhom\|\psi\|_{L^1}+(\rhoM-\rhom)\gamma_1
  \le\int_{\mathbb{T}} \psi(x-y)\rho(y)\,dy\le
  \rhoM\|\psi\|_{L^1}-(\rhoM-\rhom)\gamma_2,
\end{equation}
for any $x\in\mathbb{T}$.
Moreover, $\gamma_1$ and $\gamma_2$ can be expressed by
\[\gamma_1=\int_{\frac{\rhoM-c}{\rhoM-\rhom}}^1\psi^*(y)\,dy,\quad
\gamma_2=\int_{\frac{c-\rhom}{\rhoM-\rhom}}^1\psi^*(y)\,dy,\]
where $\psi^* : (0,1]\to\mathbb{R}$ is the decreasing rearrangement of $\psi$ on $\mathbb{T}$.
\end{lemma}

\begin{proof}
  For the lower bound, fix an $x\in\mathbb{T}$. Consider the set
  \[\mathcal{A}=\{y\in\mathbb{T} : \psi(x-y)>\psi^*(d)\},
    \quad \text{with}\quad |\mathcal{A}| = d :=\frac{\rhoM-c}{\rhoM-\rhom},\]
  and define a function
  \begin{equation}
    \label{rhostar}
    \tilde\rho (y) = \rhom \mathcal{X}_{\mathcal{A}} (y) + \rhoM \mathcal{X}_{\mathcal{A}^c} (y)=\begin{cases}\rhom&y\in\mathcal{A}\\ \rhoM&y\in \mathcal{A}^c=\mathbb{T}\backslash\mathcal{A}\end{cases},
  \end{equation}
  where $\mathcal{X}_\mathcal{A}$ denotes the indicator function of
  set $\mathcal{A}$. Let us check
  \begin{align*}
    &\int_{\mathbb{T}}\psi(x-y)\tilde\rho(y)\,dy=
   \rhom\int_{\mathcal{A}}\psi(x-y)\,dy+\rhoM\int_{\mathcal{A}^c}\psi(x-y)\,dy
    \\                                          
    &\qquad=\rhom\int_0^d\psi^*(y)\,dy + \rhoM\int_d^1\psi^*(y)\,dy=
    \rhom\|\psi\|_{L^1}+(\rhoM-\rhom)\gamma_1.
  \end{align*}
  It remains to show
$$
\int_{\mathbb{T}}\psi(x-y)(\rho(y) - \tilde\rho(y)) \,dx\geq 0, 
$$
Indeed, we have
\begin{align*}
 & \int_\mathbb{T}\psi(x-y)(\rho(y) - \tilde\rho(y))\,dy  = 
 \int_{\mathcal{A}}  \psi(x-y)(\rho(y) - \rhom)\, dy+
 \int_{\mathcal{A}^c}  \psi(x-y)(\rho(y) - \rhoM)\, dy\\
&\qquad\geq  \psi^*(d)\int_{\mathcal{A}}  (\rho(y)- \rhom)\, dy + \psi^*(d) \int_{\mathcal{A}^c}  (\rho(y) - \rhoM)\, dy\\
&\qquad=\psi^*(d) \left( \int_\mathbb{T} \rho(y)\, dy -
  \rhom\,|\mathcal{A}| - \rhoM\,|\mathcal{A}^c|\right)
  =\psi^*(d) \Big( c - \rhom d- \rhoM(1-d) \Big)= 0.
\end{align*}
The upper bound can be obtained similarly by considering the set
  \[\hat{\mathcal{A}}=\{y\in\mathbb{T} : \psi(x-y)>\psi^*(d)\},
    \quad \text{with}\quad |\hat{\mathcal{A}}| = \hat{d} :=\frac{c-\rhom}{\rhoM-\rhom},\]
  and the function
\[
    \hat\rho (y) = \rhoM \mathcal{X}_{\hat{\mathcal{A}}} (y) + \rhom \mathcal{X}_{\hat{\mathcal{A}}^c} (y).
  \]
We omit the details of the proof.
\end{proof}

\begin{remark}
The function $\tilde\rho$ in \eqref{rhostar} is the
minimizer of the optimization problem
$$
\min_{\rho}\int_{\mathbb{T}} \psi(x-y)\rho(y)\,dy,
$$
subject to the constraints in \eqref{eq:constraint}.
Indeed, the Lagrange function of the constraint minimization problem
is
$$
L(\rho,\mu_1,\mu_2,\kappa) = \int_{\mathbb{T}} \Big(\psi\rho +
\mu_1(\rho - \rhoM) - \mu_2 (\rho - \rhom) - \kappa\rho\Big) dy +
\kappa c,
$$
where $\mu_1,\mu_2\in L_+^\infty (\mathbb{T})$ and
$\kappa\in\mathbb{R}$ are Lagrange multipliers.
The Karush-Kuhn-Tucker (KKT) conditions read
\begin{align*}
& \psi + \mu_1 - \mu_2 - \kappa = 0,\quad \text{(Stationarity)},\\
& \rhom\leq \tilde\rho \leq \rhoM,\quad \int_{\mathbb{T}}\tilde \rho(y) dy = c, \quad \text{(Primal feasibility)},\\
& \mu_1,\mu_2\geq 0,\quad \text{(Dual feasibility)},\\
& (\tilde\rho - \rhoM)\mu_1 = (\tilde\rho - \rhom)\mu_2 = 0,\quad \text{(Complementary slackness)}.
\end{align*}
We choose $\mu_1 = -\min\{ 0,\psi - \kappa \} $ and $\mu_2 = \max\{
0,\psi-\kappa \}$. This ensures that the stationarity and dual
feasibility conditions are satisfied.
The complementary slackness conditions require $\tilde\rho\equiv
\rhoM$ when $\psi<\kappa$ and $\tilde\rho\equiv\rhom$ when
$\psi>\kappa$. Finally, to ensure the primal feasibility, we obtain
$\kappa=\psi^*(d)$. Altogether, we end up with \eqref{rhostar}.
\end{remark}

As a special case of Lemma \ref{psistarnewboundslem}, if we choose
$\rhoM$ and $\rhom$ as
\begin{equation}\label{eq:rhomM}
  0\le\rhom<c<\rhoM\le2c,\quad \rhom+\rhoM=2c,
\end{equation}
then \eqref{psistarnewbounds} holds with
\[\gamma_1=\gamma_2=\gamma:=\int_{\frac12}^1 \psi^*(y)\,dy.\]
It will dramatically simplify the analysis.
We also observe
\[2\gamma<\int_0^1\psi^*(y)\,dy=\|\psi\|_{L^1}.\]
In the following construction, we will choose
\begin{equation}\label{eq:rhomM2}
  \rhom=0,\quad \rhoM=2c.
\end{equation}
We shall comment that \eqref{eq:rhomM} is not the only choice that
leads to an invariant region.
We will keep using the notations $\rhoM$ and $\rhom$ throughout the
construction for generality.

\subsection{Construction of invariant region}
We shall construct the invariant region in light of $\Sigma^\ast$ as
in Section \ref{analysis1}. The main difference would be that the region
$\Sigma^\ast$ need to satisfy the additional restriction
\eqref{eq:SigmaL}.
We will make use of the improved bounds \eqref{psistarnewbounds}.
Let us denote
\begin{subequations}
\label{betaminmax}
\begin{align}
& \betaM := \rhoM\|\psi\|_{L^1} - (\rhoM-\rhom) \gamma_2. \label{betamax}\\
& \betam := \rhom\|\psi\|_{L^1} + (\rhoM-\rhom)\gamma_1,\label{betamin}
\end{align}
\end{subequations}
Unlike definition \eqref{eq:betaMm}, $\betaM$ and $\betam$ depend
on the density $\rhoM$ and $\rhom$. 
% Nevertheless, assuming \eqref{eq:rhomM}, we can easily
% check the following bounds that are independent with the density
% \begin{equation}\label{eq:betabounds}
%   2\gamma c\le\betam<\|\psi\|_{L^1}c<\betaM<2\|\psi\|_{L^1}c.
% \end{equation}

We will carry over the same notations from Section \ref{analysis1} to avoid excess notations. However, it should be noted that the functions $\hat z,\tilde z,\hat\theta,\tilde\theta$ now depend on $\rhoM,\rhom$. To avoid confusion we restate the expressions for $\hat z,\tilde z$,
\begin{align}
\label{hattildez}
& \hat z = \sqrt{\frac{4kc}{\betaM^2} - 1}, \quad 
\tilde z = \sqrt{\frac{4kc}{\betam^2} - 1}.
\end{align}
% From \eqref{eq:betabounds}, we know $\hat z,\tilde z$ are real numbers
% with the bounds
% \begin{equation}\label{eq:zbounds}
%   \sqrt{\left(\frac{\g}{2\|\psi\|_{L^1}}\right)^2-1}<\hat z < \tilde z \le \sqrt{\left(\frac{\g}{2\gamma}\right)^2-1}.
% \end{equation}
Note that if we choose $\rhoM$ and $\rhom$ as in \eqref{eq:rhomM2},
then $\hat z$ and $\tilde z$ have the explicit forms in \eqref{eq:zl1}.

Our construction of the invariant region will follow the procedure in Section \ref{analysis1}. Here we focus on the construction of the weak alignment case $\Sigma_1^\ast$. The other two cases can be treated similarly.
Let us assume $\betaM<2\sqrt{kc}$.

\textbf{Step 1:} On the $(p,q)$ plane, we construct the first segment of
the boundary of the invariant region
\begin{equation}
\label{l1c1}
C_1 = \{ (p,q): (\hat p(t),\hat q(t)), t\in [t_1,0] \},
\end{equation}
where $(\hat p, \hat q)$ satisfy the dynamics
\[
\hat p' = k-kc\hat q,\qquad
\hat q' = \hat p - \betaM\hat q,
\]
with initial data 
\[\hat p(0) = \tfrac{\betaM}{\rhoM},\quad\hat q(0) = \tfrac{1}{\rhoM}\]
and time $t_1<0$ such that $\hat q(t_1)=1/c$.
Here we choose a different initial point that Section \ref{weakal}. 
We take $q(0) = \tfrac{1}{\rhoM}$ so that $\rho(0)\le\rhoM$. The
choice of $p(0)$ ensures $\rho(t)\le\rhoM$, which is necessary for
\eqref{eq:SigmaL} to hold.

Using similar calculations as in Step 1 of Section \ref{weakal}, we have
\begin{align*}
& \hat p(t) = \frac{\betaM}{c} - e^{-\frac{\betaM t}{2}}  \left( \frac{1}{c}-\frac{1}{\rhoM} \right) \left( \betaM\cos\hat\theta t - \frac{(kc - \frac{\betaM^2}{2})}{\hat\theta}\sin\hat\theta t \right) ,\\
& \hat q(t) = \frac{1}{c} -e^{-\frac{\betaM t}{2}}\left( \frac{1}{c}  - \frac{1}{\rhoM} \right)\left( \cos\hat\theta t + \frac{\betaM}{2\hat\theta}\sin\hat\theta t \right) .
\end{align*}
The final point of $C_1$ is $(p_1, 1/c)$ where
\begin{align}
\label{l1p1}
& p_1:= \hat p(t_1) = \frac{\betaM}{c} - \sqrt{kc}\left( \frac{1}{c} - \frac{1}{\rhoM}\right) e^{\frac{\tan^{-1}\hat z}{\hat z}}.
\end{align}
The value of $p_1$ depends on the choices of $\rhoM$ and $\rhom$.

The point $(p_1,1/c)$ should be the starting point of the next segment
$C_2$. To make sure $C_2$ continues to move upward as we trace time in
the negative direction, we require that $p_1$ lies at the left
hand-side of $\frac{\betam}{c}$, which is the equilibrium state of the
auxiliary system \eqref{mumaux}.
$p_1<\frac{\betam}{c}$ can be equivalently expressed as
\begin{equation}\label{eq:p1cond}
  \betaM - \betam < \sqrt{kc}\left( 1-\frac{c}{\rhoM}\right) e^{\frac{\tan^{-1}\hat z}{\hat z}}.
\end{equation}
%This condition is implied by Lemma \ref{l1clocon2lem}.

\textbf{Step 2:} Assume that \eqref{eq:p1cond} holds.
We continue with the next segment of the boundary of the invariant region
\begin{equation}
\label{l1c2}
 C_2 = \{ (p,q): (\tilde p(t),\tilde q(t)), t\in [t_2,0] \},
\end{equation}
where $(\tilde p,\tilde q)$ satisfy the dynamics 
\[
  \tilde p' = k-kc\tilde q,\qquad
  \tilde q' = \tilde p - \betam\tilde q,
\]
with initial data $\tilde p(0) = p_1$, $\tilde q(0) = \frac{1}{c}$,
and $t_2$ is the first negative time such that $\tilde q(t_2) = 1/c$.

Using similar calculations as in Step 1 of Section \ref{weakal}, we have
\begin{align*}
& \tilde p(t) = \frac{\betam}{c}+ e^{-\frac{\betam t}{2}}\left[ \left(p_1 - \frac{\betam}{c}\right)\cos\tilde\theta t + \left( \frac{p_1\betam}{2\tilde\theta} - \frac{\beta^2_{min}}{2 c\tilde\theta}\right)\sin\tilde\theta t \right],\\
& \tilde q(t) = \frac{1}{c} + \frac{e^{-\frac{\betam t}{2}}}{\tilde\theta}\left(p_1 - \frac{\betam}{c}\right)\sin\tilde\theta t.
\end{align*}
We find that the final point of $C_2$ is $(p_2,1/c)$ where
\begin{equation}
\label{l1p2}
p_2 := \tilde p(t_2)  = \frac{\betam}{c}\left( 1+e^{\frac{\pi}{\tilde z}} \right) - p_1 e^{\frac{\pi}{\tilde z}},
\end{equation}
which also depends on the choices of $\rhoM$ and $\rhom$.

The point $(p_2,1/c)$ should be the starting point of the next segment
$C_3$. To make sure $C_3$ continues to move downward as we trace time
in the negative direction, we require that $p_2$ lies at the right
hand-side of $\frac{\betaM}{c}$, which is the equilibrium state of the
auxiliary system \eqref{muMaux}.
$p_2>\frac{\betaM}{c}$ can be equivalently expressed as 
\begin{equation}\label{eq:p2cond}
  \betaM-\betam < \sqrt{kc}\left( 1-\frac{c}{\rhoM}\right) e^{\frac{\tan^{-1}\hat z}{\hat z}}\cdot\frac{e^{ \frac{\pi}{\tilde z}} }{1+ e^{ \frac{\pi}{\tilde z}}},
\end{equation}
where $R$ is defined in \eqref{eq:p1cond}.
Note that condition \eqref{eq:p2cond} is stronger than \eqref{eq:p1cond}.

\textbf{Step 3:}
Assume that \eqref{eq:p2cond} holds. The next segment of the boundary of invariant region
\begin{equation}
\label{l1c3}
C_3 = \{ (p,q): (\hat p(t),\hat q(t)), t\in [t_3,0] \}
\end{equation}
is constructed from the dynamics 
\[
\hat p' = k-kc\hat q,\qquad
\hat q' = \hat p - \betaM\hat q,
\]
with initial data $\hat p(0) = p_2$, $\hat q(0) = \frac{1}{c}$,
and $t_3$ is the first negative time such that $\hat q(t_3) = 1/\rhoM$.
We have
\begin{align*}
& \hat p(t) = \frac{\betaM}{c} + e^{-\frac{\betaM t}{2}}\left[ \left(p_2 - \frac{\betaM}{c}\right)\cos\hat \theta t + \left(\frac{p_2\betaM}{2 \hat \theta} - \frac{\betaM^2}{2c\hat \theta}\right)\sin\hat\theta t \right],\\
& \hat q(t) = \frac{1}{c} + \frac{e^{-\frac{\betaM t}{2}}}{\hat \theta}\left(p_2 - \frac{\betaM}{c}\right)\sin \hat\theta t.
\end{align*}

To ensure the existence of $t_3$ such that $\hat q(t_3)=1/\rhoM$, we
state the following Lemma. 
\begin{lemma}
\label{l1clocon2lem}
Let $t_*$ be the first negative time such that $\hat q'(t_*) = 0$. Then $\hat q(t_*)<1/\rhoM$ if and only if 
\[
  (\betaM-\betam) \left( 1 + e^{-\frac{\pi}{\tilde z}} \right) e^{- \frac{\tan^{-1}(\hat z)}{\hat z}} < \sqrt{kc}\left( 1-\frac{c}{\rhoM} \right)\left(1- e^{-\frac{\pi}{\hat z}-\frac{\pi}{\tilde z}}  \right),
\]
or equivalently,
\begin{equation}\label{l1clocon2}
\betaM - \betam < \sqrt{kc}\left( 1-\frac{c}{\rhoM}\right) e^{\frac{\tan^{-1}\hat z}{\hat z}}\cdot \frac{(1- e^{-\frac{\pi}{\hat z}-\frac{\pi}{\tilde z}})}{( 1 + e^{-\frac{\pi}{\tilde z}})}.
\end{equation}
\end{lemma}

The proof of the Lemma follows similar arguments as Lemma
\ref{clocon2lem}, which we will omit here. The admissible condition \eqref{l1clocon2} is similar as \eqref{clocon2}, differed only by a factor, as the starting point of the construction is different.
The Lemma ensures that the trajectory of $C_3$ 
hits the line $q=1/\rhoM$ first before completing the outward spiral
turn. Moreover, at the intersection $p_3=\hat p(t_3)> \frac{\betaM}{\rhoM}$.
It is easy to observe that condition \eqref{l1clocon2} is stronger than \eqref{eq:p1cond} and
\eqref{eq:p2cond}.

Now we are ready to construct the invariant region
\begin{equation}\label{l1invreg}
 \SigmaL^\ast = \text{open set enclosed by $C_1,C_2,C_3$ and $C_4$},
\end{equation}
where $C_4$ is the line segment
\[C_4=\left\{ (p,q) : p\in(\tfrac{\betaM}{\rhoM}, p_3),\, q=\tfrac{1}{\rhoM}\right\}.\]
Figure \ref{l1invregpq} gives an illustration of the invariant region.
We can further make use of the transformation $F$ as in
\eqref{transmap1} to obtain the invariant region $\SigmaL^1$ in the $(G,\rho)$
plane
\begin{equation}\label{l1subcr}
  \SigmaL^1 := F(\SigmaL^\ast).
\end{equation}
See Figure \ref{l1subcrfig} for an illustration of $\SigmaL^1$.

\begin{figure}[ht!] 
\centering
\subfigure{\includegraphics[width=0.6\linewidth]{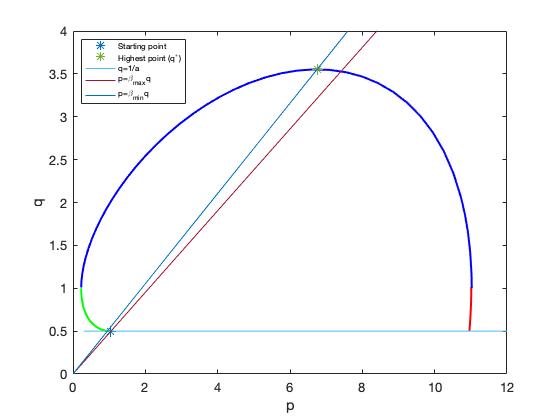}}
\caption{Invariant region for $k=4,c=1,\rhoM=2,\rhom=0$ and influence function with $\|\psi\|_{L^1} = 2, \gamma_1=\gamma_2 = 0.95$.}
\label{l1invregpq}
\end{figure}

\begin{proposition}[Invariant region]\label{l1propmain}
  Let $4kc>\betaM^2$. Assume condition \eqref{l1clocon2} holds.
  Consider the initial value problem of
  \eqref{mainodesys} with $(w(0),s(0))\in\SigmaL^\ast$. In addition,
  assume 
  \begin{equation}\label{eq:betaaprioribound}
    \betam\le\psi\ast\rho\le\betaM.
  \end{equation}
  Then the solution $(w(t),s(t))\in\SigmaL^\ast$ for all $t>0$.
\end{proposition}
\begin{proof}
The arguments that the trajectory does not cross $C_1,C_2,C_3$ are
entirely similar to the ones in the proof of Proposition
\ref{propmainwal}. If $(w,s)\in C_4$, meaning $w> \frac\betaM\rhoM$
and $s=\frac1\rhoM$, we get from \eqref{mainodesysb} that
\[s' = w-s\,(\psi\ast\rho)> \tfrac\betaM\rhoM - \tfrac1\rhoM\cdot \betaM\geq0.\]
Therefore, trajectories can not touch  trajectories with initial point
inside $\SigmaL^\ast$ never touch $C_4$ as well.
By continuity of the trajectories, we conclude that $(w(t),s(t))$
stays in $\SigmaL^\ast$ all time.
\end{proof}

For the other two cases, $\SigmaL^2,\SigmaL^3$ can be constructed very much alike as long the lines of $\Sigma_2,\Sigma_3$ respectively. The only difference is that the corresponding invariant regions on the $(p,q)$ plane now start from a point in the first quadrant, namely $(\frac\betaM\rhoM , \frac{1}{\rhoM})$, instead of the origin. Since the respective calculations and consequent proof to the second and third assertions of Theorem \ref{l1gs} follows along the lines of the first assertion, we only prove the first assertion here and state the regions $\SigmaL^2, \SigmaL^3$. 
\begin{equation}\label{l1subcr2}
\SigmaL^2 = F(\Sigma_2^\ast ),
\end{equation}
where, 
$$
\Sigma_2^\ast = \text{unbounded open set surrounded by } C_1, C_2, \{(p,0): p>\tfrac{\betaM}{\rhoM}\} \text{ on 3 sides},
$$
with $C_1$ as in \eqref{salc1} and $\hat p,\hat q$ with initial data $(\tfrac\betaM\rhoM, \tfrac1\rhoM)$ and $C_2$ as in \eqref{salc2}. Similarly, 
\begin{equation}\label{l1subcr3}
\SigmaL^3 = F(\Sigma_3^\ast ),
\end{equation}
where, 
$$
\Sigma_3^\ast = \text{open set enclosed by $C_1,C_2,C_3$ and $q=\tfrac1\rhoM$},
$$ 
with $C_1,C_2,C_3$ as in Section \ref{mediumal} but $C_1$ obtained from $\hat p,\hat q$ with initial data $(\tfrac\betaM\rhoM, \tfrac1\rhoM)$.
% \begin{figure}[h!] 
% \centering
% \subfigure{\includegraphics[width=0.4\linewidth]{l1invreg2.jpg}}
% \caption{Vector field for the system in Figure \ref{l1invregpq}.}
% \label{l1invregvectorflow}
% \end{figure}

\subsection{Proof of Theorem \ref{l1gs}}
We are ready to apply Proposition \ref{l1propmain} and prove Theorem
\ref{l1gs}. We will only prove the weak alignment case. The other two cases works similarly. 
We choose $\rhoM=2c$ and $\rhom=0$ as in
\eqref{eq:rhomM2}. It implies 
\[\betaM=2c(\|\psi\|_{L^1}-\gamma)\quad\text{and}\quad\betam=2c\gamma.\]

Let us validate all the assumptions in Proposition \ref{l1propmain}.
First, the hypothesis of the Theorem $\|\psi\|_{L^1} - \gamma <\frac\g2$
implies
\[\betaM^2=4c^2(\|\psi\|_{L^1}-\gamma)^2<4kc.\]
Second, the admissible condition \eqref{eq:wscond} implies
\eqref{l1clocon2}. Indeed, we have
\[\betaM-\betam=2c(\|\psi\|_{L^1}-2\gamma)
 < \frac{\sqrt{kc}}{2}\cdot\frac{e^{\frac{\tan^{-1}\hat z}{\hat
      z}}\left(1 - e^{ -\frac{\pi}{\tilde z} - \frac{\pi}{\hat z}
    }\right) }{\left(1 + e^{-\frac{\pi}{\tilde z}}\right)}.\]
Finally, owing to Lemma \ref{psistarnewboundslem}, we conclude that
\eqref{eq:betaaprioribound} holds as long as $\rho$ is uniformly
bounded above by $2c$ (and below by $0$).

Consider subcritical initial data $(G_0(x),\rho_0(x))\in\SigmaL^1$ for
all $x\in\mathbb{T}$. Along each characteristic path \eqref{chpath},
there is dynamics \eqref{mainodesys} with initial data $(w(0),
s(0))\in\SigmaL^\ast$. We claim that
$(w(t), s(t))\in\SigmaL^\ast$ for any $t\ge0$ along any characteristic path.

Let us argue by contradiction. Suppose there exists a first time $t_0$
and a characteristic path such that $(w(t_0),s(t_0))\not\in\SigmaL^\ast$.
By continuity of the dynamics \eqref{mainodesys}, we have that along
every characteristic path 
$(w(t_0),s(t_0))\in \overline{\SigmaL^\ast}$. Since
$\overline{\SigmaL^\ast}\subset\{(p,q) : q\ge\frac{1}{\rhoM}\}$, we
obtain the uniform bound $s(t_0,x)\geq\frac{1}{2c}$ and hence
$\rho(t_0,x)\in(0,2c]$. Now, we can apply Proposition \ref{l1propmain}
and get $(w(t_0),s(t_0))\in\SigmaL^\ast$. This leads to a contradiction.

Collecting all characteristic paths, and applying the transformation $F$ in
\eqref{transmap1}, we conclude that $(G(t,x), \rho(t,x))\in\SigmaL^1$
for all $x\in\mathbb{T}$ and $t\ge0$.
Therefore, $(G,\rho)$ remain bounded in all time.
Consequently, by Theorem \ref{local}, we have that $(\rho,u)$ is
global-in-time smooth solution to \eqref{mainsys}.

\begin{remark}
 We would like to remark the invariant region $\SigmaL^1$ is a subset of
 \[\{(G,\rho) : \tfrac1{q^\ast}\le\rho\le\rhoM\},\]
 where $q^\ast$ is the highest tip of $\SigmaL^\ast$. This leads to an improved bound on $\rhom$, and consequently better bounds on $\betam$ and $\betaM$. Repeating the procedure with the new bounds, we can obtain a larger invariant region. Finding the optimal (or largest) invariant region is beyond the scope of this paper. We shall leave this for future investigations.
\end{remark}

\section*{Acknowledgments}
This research was supported by the National Science Foundation under
grants DMS18-12666 (MB and HL), DMS18-53001 and DMS21-08264 (CT).

\bigskip

\bibliographystyle{abbrv}

\begin{thebibliography}{10}
%\begin{thebibliography}{20}
\bibitem{BL19}
 M. Bhatnagar and H. Liu.
 \newblock Critical thresholds in one-dimensional damped Euler-Poisson systems.
 \newblock {\em Math. Mod. Meth. Appl. Sci.}, 30(5): 891--916, 2020. 
 
 \bibitem{BL201}
 M. Bhatnagar and H. Liu.
 \newblock Critical thresholds in 1D pressureless Euler-Poisson systems with variable background.
 \newblock {\em Physica D: Nonlinear Phenomena}, 414: 132728, 2020. 
 
 \bibitem{BL202}
 M. Bhatnagar and H. Liu.
 \newblock Well-posedness and critical thresholds in nonlocal Euler system with relaxation.
 \newblock {\em Disc. Cont. Dyn. Sys.}, 41(11): 5271--5289, 2021. 	
 
 \bibitem{BLws}
 M. Bhatnagar and H. Liu.
 \newblock Global dynamics of the Euler-alignment system with weakly singular kernel.
 \newblock {\em ArXiv: 2110.10314}, 2021. 	
 
 \bibitem{BRR94}
 U. Brauer, A. Rendall and O. Reula.
 \newblock The cosmic no-hair theorem and the non-linear stability of homogeneous Newtonian cosmological models.
 \newblock {\em Classical and Quantum Gravity}, 11(9): 2283, 1994. 	
 

\bibitem{CCTT16}
 J. A. Carrillo, Y.-P. Choi, E. Tadmor, and C. Tan.
 \newblock Critical thresholds in 1D Euler equations with non-local forces.
 \newblock {\em  Math. Mod. Meth. Appl. Sci.}, 26:185--206, 2016.
 
 
\bibitem{CCZ16}
 J.A. Carrillo, Y.P. Choi, E. Zatorska.
 \newblock On the pressureless damped Euler-Poisson equations with quadratic confinement.
 \newblock {\em  Math. Mod. Meth. Appl. Sci.}, 26: 2311-2340, 2016.
 
 \bibitem{CS07a}
 F. Cucker and S. Smale.
 \newblock Emergent Behavior in flocks.
 \newblock {\em  IEEE Transactions on Automatic Control}, 52(5): 852--862, 2007.
 
 \bibitem{CS07b}
 F. Cucker and S. Smale.
 \newblock On the mathematics of emergence.
 \newblock {\em  Japanese Journal of Mathematics}, 2(1): 197--227, 2007.

\bibitem{Da16}
C.M. Dafermos.
 \newblock Hyperbolic Conservation Laws in Continuum Physics.
 \newblock {\em  Springer-Verlag Berlin Heidelberg}, Vol. 325, 2010 (3rd Ed.).

\bibitem{DKRT18}
 T. Do, A. Kiselev, L. Ryzhik, C. Tan.
 \newblock Global regularity for the fractional Euler alignment system.
 \newblock {\em  Arch. Rat. Mech. Anal.}, 228(1): 1--37, 2018.
 
\bibitem{ELT01}
 S. Engelberg, H. Liu, E. Tadmor.
 \newblock Critical thresholds in Euler-Poisson equations.
 \newblock {\em  Indiana University Math. Journal}, 50:109--157, 2001.

% \bibitem{FV16} 
%L. C. F. Ferreira,  J. C.Valencia-Guevara.
%\newblock Periodic solutions for a 1D-model with nonlocal velocity via mass transport.
% \newblock {\em J. Diff. Equ.},  260: 7093--7114. 

 \bibitem{GM62} 
 S. R. de Groot, and P. Mazur.
 \newblock Non-Equilibrium Thermodynamics. 
  \newblock {\em North-Holland Publishing Company}, Amsterdam, 1962.
  
  \bibitem{HaTa08}
 S.Y. Ha and E. Tadmor.
 \newblock From particle to kinetic and hydrodynamic descriptions of flocking.
 \newblock {\em  Kinetic and Related Models}, 1:415--435, 2008.
 
%\bibitem{GT01}
%D. Gilbarg, N. S. Trudinger.
% \newblock Elliptic Partial Differential equations of second order.
% \newblock {\em  Springer-Verlag Berlin Heidelberg}, Vol. 224, 2001 (2nd Ed.).

\bibitem{HJL81}
D.D. Holm, S.F. Johnson and K.E. Lonngren.
 \newblock  Expansion of a cold ion cloud.
 \newblock {\em Applied Physics Letters}, 38(7): 519--521, 1981. 
 
 \bibitem{HT17}
S. He,  E. Tadmor.
 \newblock  Global regularity of two-dimensional flocking hydrodynamics.
 \newblock {\em C. R. Math.}, 355(7): 795--805, 2017. 
 
  \bibitem{Ja75}
J.D. Jackson.
 \newblock  Classical electrodynamics.
 \newblock {\em Wiley}, 1975. 
  
\bibitem{KT18}
 A. Kiselev, C. Tan.
 \newblock Global regularity for 1D Eulerian dynamics with singular interaction forces.
 \newblock {\em  SIAM J. Math. Anal.}, 50(6):6208--6229, 2018.
 
\bibitem{Lax64}
P. Lax.
 \newblock Development of singularities of solutions of nonlinear hyperbolic partial differential equations.
 \newblock {\em  Journal of Math. Phys.}, 5, 611, 1964.
 
 \bibitem{Lee2}
Y. Lee and H. Liu.
 \newblock Thresholds in three-dimensional restricted Euler-Poisson equations.
 \newblock {\em Physica D: Nonlinear Phenomena}, 262: 59--70, 2013.

\bibitem{LL08}
 T. Li, H. Liu.
 \newblock Critical Thresholds in a relaxation model for traffic flows.
 \newblock {\em Indian Univ. Math. J.}, 57:1409--1431, 2008.
 
 
\bibitem{LL09j}
 T. Li, H. Liu.
 \newblock Critical Thresholds in a relaxation system with resonance of characteristic speeds.
 \newblock {\em  Disc. Cont. Dyn. Sys-Series A}, 24(2):511--521, 2009.
 
\bibitem{LL09}
 T. Li, H. Liu.
 \newblock Critical thresholds in hyperbolic relaxation systems.
 \newblock {\em  J. Diff. Eqns.}, 247:33--48, 2009.

 
 \bibitem{LT01}
 H. Liu, E. Tadmor.
 \newblock Critical thresholds in a convolution model for nonlinear conservation laws.
 \newblock {\em  SIAM J. Math. Anal.}, 33(4): 930--945, 2001.
  

\bibitem{LT02} 
H. Liu and E. Tadmor.
 \newblock  Spectral dynamics of the velocity gradient field in restricted flows.
 \newblock {\em Commun.  Math. Phys.}, 228: 435--466, 2002.

\bibitem{LT03}
 H. Liu, E. Tadmor.
 \newblock Critical Thresholds in 2-D restricted Euler-Poisson equations.
 \newblock {\em SIAM J. Appl. Math.}, 63(6):1889--1910, 2003.
 
 
\bibitem{LT04}
 H. Liu, E. Tadmor.
 \newblock Rotation prevents finite-time breakdown.
 \newblock {\em  Physica D}, 188:262--276, 2004.
 
 \bibitem{Ma86}
 T. Makino.
 \newblock On a local existence theorem for the evolution equation of gaseous stars.
 \newblock {\em  Studies in Mathematics and its Applications}, 18: 459--479, 1986.
 
 \bibitem{MaPe90}
 T. Makino and B. Perthame.
 \newblock Sur les solution {\`a} sym{\'e}trie sph{\'e}rique de l’equation d’Euler-Poisson pour l’evolution d’etoiles gazeuses.
 \newblock {\em  Japan Journal of Applied Mathematics}, 7(1): 165--170, 1990.
 
 \bibitem{MRS90}
 P.A. Markowich, C.A. Ringhofer and C. Schmeiser.
 \newblock Semiconductor equations.
 \newblock {\em  Springer-Verlag}, 1st edition, 1990.


%\bibitem{M84}
%A. Majda.
% \newblock Compressible Fluid Flow and Systems of Conservation Laws in Several Space Variables.
% \newblock {\em  Springer Science+Business Media}, Vol. 53, 1984.
 

% \bibitem{PRS90}
%P.A. Markowich, C.A. Ringhofer, and C. Schmeiser. 
%\newblock Semiconductor Equations. 
% \newblock {\em Springer- Verlag,} 1990.
 
\bibitem{MTX21}
 Q. Miao, C. Tan, and L. Xue.
 \newblock Global regularity for a 1D Euler-alignment system with misalignment.
 \newblock {\em Math. Mod. Meth. Appl. Sci.}, 31(3): 473--524, 2021. 

 
%\bibitem{RLM96} 
%G. R. Baker, X. Li,  and A. C. Morlet.
% \newblock  Analytic structure of two 1D-transport equations with fluxes.
% \newblock {\em  Physica D},  91: 349--375, 1996.

\bibitem{Shv21}
  R. Shvydkoy.
  \newblock Dynamics and analysis of alignment models of collective behavior.
  \newblock {Birkhäuser}, 2021.
  
\bibitem{ST17}
  R. Shvydkoy, E. Tadmor.
  \newblock Eulerian dynamics with a commutator forcing.
  \newblock {\em Trans. Math. and Appl.}, 1(1): tnx001, 2017.
 
\bibitem{TT14}
 E. Tadmor, C. Tan.
 \newblock Critical thresholds in flocking hydrodynamics with non-local alignment.
 \newblock {\em  Phil. Trans. R. Soc. A.}, 372: 20130401, 2014.

\bibitem{TW08}
 E. Tadmor, D. Wei.
 \newblock On the global regularity of subcritical Euler-Poisson equations with pressure.
 \newblock {\em  J. Eur. Math. Soc.}, 10:757--769, 2008.

\bibitem{Tan20}
C. Tan.
 \newblock On the Euler-alignment system with weakly singular communication weights.
 \newblock {\em  Nonlinearity}, 33(4): 1907--1924, 2020.

 \bibitem{Tan21}
C. Tan.
 \newblock Eulerian dynamics in multi-dimensions with radial symmetry.
 \newblock {\em  SIAM J. Math. Anal.}, 53(3): 3040--3071, 2021.

 
 \bibitem{Wa19}
 W.A. Yong.
\newblock Intrinsic properties of conservation-dissipation formalism of irreversible thermodynamics.
 \newblock {\em Phil. Trans. R. Soc. A} 378:20190177, 2020.  
 
\bibitem{WTB12}
 D. Wei, E. Tadmor, H. Bae.
 \newblock Critical Thresholds in multi-dimensional Euler-Poisson equations with radial symmetry.
 \newblock {\em  Commun. Math. Sci.}, 10(1):75--86, 2012.		


\end{thebibliography}

\end{document}